\documentclass[letterpaper]{scrartcl}

\usepackage[USenglish]{babel}
\usepackage[utf8]{inputenc}

\usepackage{geometry}
\usepackage{mathtools}
\usepackage{amssymb}
\usepackage{amsthm}
\usepackage{graphicx}
 
\usepackage[labelformat=simple]{subcaption}

\usepackage[shortlabels]{enumitem}

\usepackage[linesnumbered,ruled,vlined]{algorithm2e}
\SetArgSty{text}

\usepackage{authblk}
\title{The Cheapest Ticket Problem in Public Transport\footnote{This work was partially supported by JPI Urban Europe under the project EASIER.\\Parts of the paper are based on a preliminary version presented at ATMOS 2020 \cite{SchUrb-Atmos}.}}
\author[1,2]{Anita Schöbel}
\author[1]{Reena Urban}
\date{}
\affil[1]{Technische Universität Kaiserslautern, Germany}
\affil[2]{Fraunhofer-Institute for Industrial Mathematics ITWM, Germany}
\affil[ ]{\textit {\{schoebel,urban\}@mathematik.uni-kl.de}}

\newcommand{\N}{\mathbb{N}}
\newcommand{\R}{\mathbb{R}}
\newcommand{\W}{\mathcal{W}}
\newcommand{\cZ}{{\mathcal{Z}}}
\newcommand{\defeq}{\mathrel{\vcentcolon=}}
\newcommand{\smax}{S_{\max}}
\newcommand{\lmax}{L_{\max}}
\newcommand{\dmax}{D_{\max}}
\newcommand{\kmax}{k_{\max}}
\newcommand{\zones}{\textrm{zones}}

\theoremstyle{plain}
\newtheorem{theorem}{Theorem}
\newtheorem{lemma}[theorem]{Lemma}

\newtheorem{corollary}[theorem]{Corollary}
\theoremstyle{definition}
\newtheorem{definition}[theorem]{Definition}
\newtheorem{example}[theorem]{Example}

\graphicspath{{Figures/}}

\usepackage{abstract}

\usepackage[hidelinks]{hyperref}
\begin{document}

\maketitle

\begin{abstract}
	Route choice models in public transport have been discussed for a long time. The main factor why a passenger chooses a specific path is usually based on its length or travel time.
	However, also the ticket price that passengers have to pay may influence their decision since passengers prefer cheaper paths over more expensive ones.
	
	In this paper, we deal with the \emph{cheapest ticket problem} which asks for a cheapest ticket to travel between a pair of stations.
	The complexity and the algorithmic approach to solve this problem depend crucially on the underlying \emph{fare structure}, e.g., it is easy if the ticket prices are proportional to the distance traveled (as in distance tariff fare structures), but may become NP-complete in zone tariff fare structures.
	We hence discuss the cheapest ticket problem for different variations of distance- and zone-based fare structures.
	We start by modeling the respective fare structure mathematically, identify its main properties, and finally provide a polynomial algorithm, or prove NP-completeness of the cheapest ticket problem. 
	We also provide general results on the combination of two fare structures, which is often observed in practice.\\
\end{abstract}

	
\section{Introduction}
Fare systems may be very diverse, containing a lot of different rules and regulations.
Among the possible fare strategies is the flat tariff in which all journeys cost the same, no matter how long they are, or kilometer-based distance tariffs which are used by most railway companies all over the world.
Very popular in metropolitan regions are zone tariffs (used in many European cities, but also, e.g., in California) in which the number of zones traversed on a journey determines the ticket price. 
In most regions, these fare strategies come with special regulations: journeys with less than a given number of stations may get a special price, there might be network-wide tickets or stations belonging to more than one zone.
Sometimes different fare strategies are combined.
The underlying fare strategy is usually independent of the way tickets are bought: they can be provided as paper tickets from ticket machines or from online sales, by usage of smart cards in check-in-check-out systems, or by other mobile devices.
Recently, some public transport companies offer the simple usage of a mobile device for charging the beeline tariff between the start coordinates and the end coordinates of the journey. 

The question which we pursue in this paper is how to find the cheapest possibility to travel between two stations if the fare structure, i.e., the fare strategy with prices, is known.
In particular, we want to determine in which cases the cheapest ticket problem is solvable in polynomial time with shortest path techniques.
This question is relevant for several reasons. 
First, the passengers would like to minimize the ticket prices they have to pay as one among other criteria when planning their journeys. 
Second, a public transport company can only estimate its income if the ticket prices are specified. 
Knowing the demand and the cheapest ticket prices, hence, gives a lower bound on the
expected income of the public transport provider.
Third, for designing and improving fare structures, it is necessary to be able to compute (cheapest) ticket prices. 
Last, knowing about the ticket prices and combining them with the expected travel time may help to understand the passengers' behavior better and hence may lead to more realistic passenger assignment models.

In this paper, we discuss the following two properties:
\begin{description}
	\item[No-stopover property:] Passengers cannot save money by splitting a journey into two (or more) parts and buying separate tickets for each of these sub-journeys.
	\item[No-elongation property:] Passengers cannot save money by buying a ticket for a longer journey but only using a part of it.
\end{description}

As will be shown, there is no relation between these two properties. 
They are relevant from a real-world point of view since they ensure that a fare structure is consistent and does not trigger strange actions (e.g., buying a ticket for a longer path than needed) as a legal way of saving money.
In \cite{HannSchnee05}, the authors say that a fare structure without the no-stopover property would be ``impractical and potentially confusing for the customer''.
Still, this property is not always satisfied in real-world fare structures, see \cite{Urban20}. 
The no-elongation property is taken into account in \cite{OttoBoysen17}.\medskip

As already mentioned, the cheapest ticket problem depends crucially on the underlying fare structure. 
Literature on fare structures is scarce compared to papers on timetabling or scheduling in public transport.
Early papers deal with the design of zone tariffs, see \cite{HaScho1,HaSc01,KeBa01,BKP05}.
The topic is still ongoing using different types of objectives, e.g., the income of the public transport company,
as in \cite{BornKarbPfetsch12,GalliMaischSchoen17,OttoBoysen17}.
Also the design of distance tariffs from zone tariffs is studied in \cite{MaadiSchmoecker18}.
The computation of cheapest paths is considered for distance tariffs in a railway context in \cite{HannSchnee05}, while \cite{DellingPajorWerneck} and \cite{DellingDibbeltPajor} compute paths that traverse the smallest number of tariff zones.
\cite{Blanco2016,Blanco2017} deal with cost-minimal paths in the context of flight trajectory optimization with overflight costs.
Other papers, e.g., \cite{Lo2003} and \cite{Maadi2017}, in the area of route choice and multi-modal routing deal with transfers, non-additive link costs and hyperpaths.
Recently, \cite{EulBorn19} have presented the so-called ticket graph which models transitions between tickets via transition functions over partially ordered monoids and allows the design of an algorithm for finding cheapest paths in fare structures which do not have the subpath-optimality property. 
However, the runtime of this approach need not be polynomial.\medskip

The remainder of the paper is structured as follows: 
In Section~\ref{sec-faresystems}, we start by defining what a fare structure is, and we define the no-stopover property and the no-elongation property formally.
We then discuss these properties for different fare structures and develop algorithms and/or complexity results regarding the cheapest ticket problem. 
This is done for distance tariffs, beeline tariffs and flat tariffs in Section~\ref{sec-distance}, for various zone-based fare structures in Section~\ref{section zone}, and for the combination of different fare structures in Section~\ref{chapter combined}.
For a more detailed list, see Table~\ref{Table:overview-fare-structures}.
We conclude in Section~\ref{sec-conclusion}.

\begin{table}
	\caption{Overview of the discussed fare strategies.}
	\label{Table:overview-fare-structures}
	\begin{small}
		\begin{center}
			\begin{tabular}{|ll|}
				\hline
				flat/distance/beeline tariff & Section~\ref{sec-distance}\\ \hline
				basic zone tariff & Section~\ref{section basic zone}\\ \hline
				zone tariff with metropolitan zone & Section~\ref{sec-metropolitan}	\\ \hline
				zone tariff with overlap areas & Section~\ref{sec-overlap}		\\ \hline
				zone tariff with single counting & Section~\ref{sec-single}		\\ \hline
				combined fare structures (general properties) & Section~\ref{sec-combined} \\ \hline
				bounded distance tariff & Section~\ref{sec-bounded} \\ \hline
				basic zone tariff combined with a short-distance tariff & Section~\ref{section-zone-short} \\ \hline	
			\end{tabular}
		\end{center}
	\end{small}
\end{table}

\section{Fare Structures in Public Transport}
\label{sec-faresystems}
In terms of terminology, we refer to \cite[p.\ 13]{Fleishman1996}. 
By \emph{fare strategy} we describe the general type of payment structure, e.g., a flat tariff, a distance tariff or a zone tariff and their particularities.
The \emph{fare structure} is the combination of one or more fare strategies with actual prices.
We consider two types of fare strategies, namely distance-based and zone-based fare strategies, which we define in their dedicated sections.
Besides these fare strategies which are popular in many countries, there are further possibilities not covered here (see, e.g., \cite{Schmoecker2016}).
We first specify mathematically what a fare structure is. 
We are not aware of such a formal definition in the literature. 

Let a \emph{public transport network (PTN)} be given. 
A PTN $(V,E)$ is a graph with a node set $V$ given by a set of stops or stations and an edge set $E$ of direct connections between them.
For simplicity, we assume the PTN to be an undirected graph which is simple and connected. 
A subset of nodes $Z\subseteq V$ is called connected if its induced subgraph $G[Z]$ is connected.
The PTN can be used to model railway, tram, or bus networks. 
In the following, we call the nodes of the PTN \emph{stations} also if bus networks with stops are under consideration.
The price of a journey through a PTN depends not only on the start station and the end station of the journey, but also on the specific path and the tickets that have been chosen.
Here, a path is a finite sequence of (not necessary distinct) nodes and edges.
Since the PTN is simple, a path is uniquely determined by its sequence of nodes.

\begin{definition}
	Let a PTN be given, and let $\W$ be the set of all paths in the PTN.
	A \emph{fare structure} is a function $p:\W \to \R_{\geq 0}$ that assigns a price to every path in the PTN.
\end{definition}

For a path $W=(x_1,\ldots,x_n)$, we denote a subpath $(x_i,\ldots,x_j)$ with $1 \leq i \leq j \leq n$ by $[x_i,x_j]$. 
The price of a subpath is hence given as $p([x_i,x_j])$.
The brackets $[ \cdot ]$ emphasize that $[x_i,x_j]$ describes a path and not only a pair of stations.
Furthermore, let $W_1+W_2$ denote the concatenation of paths $W_1$ and $W_2$.

Next, we specify with which combinations of tickets a passenger can travel (legally) along a given path $W$:
The straightforward way is to pay for the whole path $W\!$. But it may also be possible to pay for a longer path $H_1 \supseteq W$ and use the ticket only for the subpath $W\!$, or to split $W$ into $W_1+W_2$ and use two tickets, namely for $W_1$ and for $W_2$ instead of one ticket for the whole path~$W\!$. 
The general definition is the following.

\begin{definition} \label{def-ticket}
	A finite sequence of paths $T=(H_1,\ldots,H_t)$ with $H_j \in \W$, $j\in \{1,\ldots,t\}$, is a \emph{ticket} of a path $W \in \W$ if there is a partition of $W$ into subpaths $W_1,\ldots, W_t$ such that ${W = W_1+\ldots+W_t}$ and $W_j$ is a subpath of $H_j$, $j \in \{1\ldots, t\}$.
	The \emph{price} of a ticket $T = (H_1,\ldots,H_t)$ is given by ${p(T) \defeq \sum_{j=1}^t p(H_j)}$.
\end{definition}

For a given path $W\!$, the ticket $T=(W)$ is the \emph{standard ticket}, $T=(H_1)$ with $W \subseteq H_1$ is an \emph{elongated ticket} and $T=(W_1,W_2)$ with $W_1+W_2=W$ is a \emph{compound ticket}.

Our goal is to find cheapest tickets and cheapest standard tickets to travel from a station $x$
to another station $y$.
A path with a cheapest standard ticket from $x$ to $y$ is called a \emph{cheapest path} between those stations.

\begin{example} \label{example-ticket}
	\begin{figure}
		\centering
		\includegraphics[scale=1]{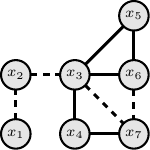}
		\caption{PTN for Example~\ref{example-ticket}.}
		\label{Figure-Example-Ticket}
	\end{figure}
	In order to illustrate the previous definition, we consider the PTN depicted in Figure~\ref{Figure-Example-Ticket}.
	Let $W=(x_1,x_2,x_3,x_7,x_6)$ (dashed) be the path along which we travel.
	Then there are several feasible tickets.
	For example, the standard ticket is given by $T_1=(W)$, the ticket ${T_2=(W+(x_6,x_5))}$ is an elongated ticket, $T_3=((x_1,x_2,x_3),(x_3,x_7,x_6))$ is a compound ticket, and a combination of both is given by $T_4=((x_1,x_2,x_3,x_4), (x_5, x_3,x_7), (x_7,x_6))$, where $W_1=(x_1,x_2,x_3)$, $W_2=(x_3,x_7)$ and $W_3=(x_7,x_6)$.
	The last case could, for example, be relevant if $x_5-x_3-x_7$ is a railway line with an especially cheap ticket.
\end{example}

Note that Definition~\ref{def-ticket} can be further generalized, e.g., by allowing $W_i,W_j \subseteq H_k$
for the same path $H_k$ (which then has to be paid only once) or by switching the order of the $H_k$.
It can be shown that these relaxations can be neglected when searching for cheapest tickets between
two stations as long as the no-elongation property holds.

We now define the no-stopover property and the no-elongation property formally.
Let a fare structure~$p$ on a PTN $(V,E)$ be given.

\begin{definition}
	\label{no-stopover}
	A fare structure $p$ satisfies the no-stopover property if 
	\[ p([x_1,x_n]) \leq p([x_1,x_i])+p([x_i,x_n]) \]
	for all paths $(x_1,\ldots,x_n) \in \W$, $n\geq 3$, and all intermediate stations $x_i$ with $i \in \{2,\ldots,n-1\}$.
\end{definition}

The no-stopover property states that a compound ticket $([x_1,x_i], [x_i,x_n])$ is never preferable to the standard ticket $([x_1,x_n])$, i.e., splitting the ticket or making a stopover does never decrease the ticket price.
Also multiple stopovers of a single path are not favorable, since for several stopovers
at $x_{i_1},\ldots,x_{i_k}$, we have that
\[
\underbrace{\underbrace{p([x_1,x_{i_1}]) + p([x_{i_1},x_{i_2}])}_{\geq p([x_1,x_{i_2}])} +p([x_{i_2},x_{i_3}])}_{\geq([x_1,x_{i_3}])}
+ \cdots + p([x_{i_k},x_n]) \geq p([x_1,x_n]). 
\]

We acknowledge that there might be very special situations in which stopovers are intended.
For example, in order to incentivize shopping at a place along a commuter route, one could offer a cheaper conjunction ticket which is valid in combination with a receipt after shopping.
However, this is complicated to control and does not concern the fundamental fare structure.

\begin{definition}
	A fare structure $p$ satisfies the no-elongation property if ${p([x_1,x_{n-1}])\leq p([x_1,x_n])}$ for all paths $(x_1,\ldots,x_n) \in \W$, $n\geq 2$.
\end{definition}  

The no-elongation property states that, given a path, an elongated ticket is not preferable to a standard ticket, i.e., $p(W)\leq p(H_1)$ for $W\subseteq H_1\in \W$.
This is because for every path $(x_1,\ldots,x_n)$ with a subpath $[x_i,x_j]$, $1\leq i\leq j\leq n$, we have $p([x_1,x_n]) \geq p([x_1,x_{n-1}]) \geq \ldots \geq p([x_1,x_j])$ and ${p([x_1,x_j])\geq p([x_2,x_j])\geq p([x_i,x_j])}$ by considering the reverse paths.

While the absence of the no-stopover and no-elongation property is mostly seen as non-transparent (maybe even unfair), there are a few situations in which these properties are not intended, e.g., to make leaving a train at an airport more costly.
If conventional tickets are used, passengers could avoid paying such an increased price by buying a ticket for a longer journey and leaving the train early, while there is no legal way around in check-in/check-out systems.

In Section~\ref{sec-distance}, we will see that the no-stopover property does not imply the no-elongation property, and Section~\ref{sec-metropolitan} will show that the inverse implication does also not hold.\medskip

Observe that the price of a cheapest ticket can be smaller than the price of a cheapest standard ticket.
However, if the no-stopover and the no-elongation property both hold, both problems are equivalent in the following sense:

\begin{theorem}\label{Prelim-thm}
	If a fare structure satisfies the no-stopover and the no-elongation property, then the standard ticket $T=(W)$ is a cheapest ticket for every path $W \in \W$.
	
	In particular, a cheapest standard ticket between two stations is also a cheapest ticket.
\end{theorem}
\begin{proof}
	Let a fare structure $p$ be given for which the no-stopover and the no-elongation property are satisfied.
	Further, let $T =(H_1,\ldots,H_t)$ be a ticket for the path $W \in \W$.
	We show that the standard ticket of $W$ does not cost more than $T$.
	
	Let $W=W_1 +\ldots + W_t$ be a decomposition such that $W_j$ is a subpath of $H_j$, $j\in \{1,\ldots, n\}$.
	Due to the no-elongation property, we have that $p(W_j)\leq p(H_j)$.
	The standard ticket of $W$ is ${T' = (W) = (W_1 + \ldots + W_t)}$.
	Due to the no-stopover property, it holds that 
	\[p(T') = p(W_1 + \ldots + W_t)\stackrel{\text{no-stopover}}{\leq} p(W_1) + \ldots + p(W_t)
	\stackrel{\text{no-elongation}}{\leq} p(H_1) + \ldots p(H_t)=p(T).\]
	Hence, the standard ticket $T'$ of $W$ is at least as cheap as the ticket $T$.
\end{proof}

In other words, if both the no-elongation and the no-stopover property hold, there always exists a cheapest possibility to travel from $x$ to $y$ which can be realized by a standard ticket, i.e., by a cheapest path $[x,y]$ for which the passenger buys one single ticket with price $p([x,y])$.
Consequently, under the assumptions of the theorem, we only need to consider standard tickets in order to determine a cheapest ticket. 
This will be used later on and simplifies the situation.

\section{Distance-based Fare Structures}
\label{sec-distance}
For distance-based fare structures, the price of a journey depends on the kilometers traveled. 
In  a \emph{distance tariff} the length of the path is used, while in a \emph{beeline tariff} the distance as the crow flies (Euclidean distance) is the basis for the ticket price. 
We use $l(W)$ to denote the length of a path ${W=(x_1,\ldots,x_n)}$ (in kilometers), and $l_2(W)=\|x_n - x_1\|_2$ as its beeline distance.
In order to compute $l(W)$, we assume that each edge in the PTN has assigned its (positive) physical length.
For the beeline distance, we assume that the stations~$V$ of the PTN are embedded in the plane such that the Euclidean distance $l_2$ between every pair of stations can be computed.

\begin{definition}
	Let a PTN be given, and let $\W$ be the set of all paths in the PTN.
	A fare structure~$p$ is a \emph{distance tariff} w.r.t.\ a price function $P \colon \R_{\geq 0} \to \R_{\geq 0}$ if $p(W)=P(l(W))$ for all~${W \in \W}$.
\end{definition}

\begin{definition}
	Let a PTN be given, and let $\W$ be the set of all paths in the PTN.
	A fare structure~$p$ is a \emph{beeline tariff} w.r.t.\ a price function $P \colon \R_{\geq 0} \to \R_{\geq 0}$ if $p(W)=P(l_2(W))$ for all~${W \in \W}$.
\end{definition}

A common case is an \emph{affine} price function ${P \colon \R_{\geq 0} \to \R_{\geq 0}}$, $k \mapsto P(k)\defeq f + \bar{p} \cdot k$ for some base amount $f \geq 0$ and a price per kilometer $\bar{p} \geq 0$.
Note that in case of a constant price function~$P$ (i.e., if $P$ is affine with $\bar{p}=0$), both the distance tariff and the beeline tariff become a \emph{flat tariff} (also called \emph{unit tariff}), in which all paths cost the same.
Most railway systems rely on distance tariffs (or modifications). 
Beeline tariffs are rather new and often used for mobile tickets on mobile phones which track the journey of a passenger by using her GPS coordinates and determining the price based on the beeline distance after the journey is over.\medskip

We start by analyzing the distance tariff.
\begin{theorem} \label{distance-eig1}
	Let a price function $P$ be given.
	All distance tariffs w.r.t.\ $P$ satisfy the no-stopover property if and only if $P$ is subadditive on $\R_{> 0}$.
\end{theorem}
\begin{proof}
Consider a path $W \in \W$ with a corresponding compound ticket $(W_1,W_2)$.
We have $l(W)=l(W_1)+l(W_2)$.
If $P$ is subadditive, i.e., $P(a+b)\leq P(a)+P(b)$ for all $a,b \in \R_{> 0}$, it holds that $p(W)=P(l(W)) \leq P(l(W_1))+ P(l(W_2))=p(W_1)+p(W_2)$.

Now assume that $P$ is not subadditive, i.e., there are $a,b \in \R_{> 0}$, with $P(a+b) >P(a)+ P(b)$.
Consider a line graph with three stations $x_1, x_2, x_3$ such that $l(x_1,x_2)=a$, $l(x_2,x_3)=b$.
Then $p((x_1,x_2,x_3))=P(a+b)>P(a)+P(b)=p((x_1,x_2))+p((x_2,x_3))$ and the no-stopover property is not satisfied.
\end{proof}

\begin{theorem}\label{distance-eig2}
	Let a price function $P$ be given.
	All distance tariffs w.r.t.\ $P$ satisfy the no-elongation property if and only if $P$ is increasing. 
\end{theorem}
\begin{proof}
Note that for a path $W=(x_1,\ldots,x_n)\in \W$, $n\geq 2$, we have that $l([x_1,x_{n-1}])\leq l(W)$. Hence, if $P$ is increasing, then $p([x_1,x_{n-1}]) = P(l([x_1,x_{n-1}])) \leq P(l(W)) = p(W)$.

Now assume that $P$ is not increasing, i.e., there are $a,b \in \R_{\geq 0}$, $a<b$, with $P(a) > P(b)$.
Consider a line graph with three stations $x_1, x_2, x_3$ such that $l(x_1,x_2)=a$, $l(x_2,x_3)=b-a$ (if $a=0$, let $x_1=x_2$).
Then $p((x_1,x_2))=P(a)>P(b)=p((x_1,x_2,x_3))$, and the no-elongation property is not satisfied.
\end{proof}

For a distance tariff, we have for any pair of paths $W_1,W_2 \in \W$ that $l(W_1) \leq l(W_2)$ is equivalent to ${p(W_1) \leq p(W_2)}$ if the price function is increasing.
Hence, a shortest path is always a cheapest path in this case, and vice versa. 
Therefore, we can use any shortest path algorithm in order to compute a cheapest path.
By Theorems~\ref{Prelim-thm}, \ref{distance-eig1} and~\ref{distance-eig2}, cheapest standard tickets are also cheapest tickets in a distance tariff if the price function is subadditive and increasing.
In particular, this is also true for flat tariffs and distance tariffs with an affine price function.

\begin{corollary}\label{distance-cor}
		Let $p$ be a distance tariff with an increasing price function $P\!$.
		\begin{itemize}
			\item A cheapest path and hence a cheapest standard ticket can be computed in polynomial time.
			\item If $P$ is subadditive on $\R_{>0}$, then a cheapest ticket can be computed in polynomial time.
		\end{itemize}
\end{corollary}

For the beeline tariff, we use the Euclidean distance to determine the price of a ticket.
This means that the ticket price is only dependent on the location of the start and end station, but not on the specific path to travel between them. 
Consequently, \emph{all} paths between two stations $x$ and~$y$ are cheapest paths and hence can be found in polynomial time, e.g., by breadth-first search.

\begin{lemma}
	For a beeline tariff, a cheapest standard ticket can be computed in polynomial time.
\end{lemma}

In general, a beeline tariff does not satisfy the no-elongation property.
For the no-stopover property, we need a stronger assumption, namely $P(a) \leq P(b) + P(c)$ for all $a,b,c\in \R_{\geq 0}$, $a \leq b+c$.
This holds if $P$ is subadditive and increasing.

\begin{theorem}
	Let a price function $P$ be given.	
	All beeline tariffs w.r.t.\ $P$ satisfy the no-stopover property if and only if $P(a) \leq P(b) + P(c)$ for all $a,b,c\in \R_{> 0}$ with $a \leq b+c$, and $P(0)\leq 2 P(d)$ for all $d\in \R_{> 0}$.
\end{theorem}

\begin{proof}
	First, we assume that there are $a,b,c \in \R_{> 0}$, $a\leq b+c$, with $P(a)>P(b)+P(c)$.
	Consider a complete graph with three stations $x_1,x_2,x_3$ such that $l_2(x_1,x_3)=a$, $l_2(x_1,x_2)=b$ and ${l_2(x_2,x_3)=c}$.
	Then $p((x_1,x_2,x_3))=P(a)>P(b)+P(c)=p((x_1,x_2))+p((x_2,x_3)) $ and the no-stopover property is not satisfied.
	If there is some $d\in \R_{>0}$ with $P(0)>2P(d)$, then consider the path $W=(x_1,x_2,x_1)$ with $l_2(x_1,x_2)=d$.
	The no-stopover property is not satisfied because ${p(W)=P(0)>2P(d)=p((x_1,x_2))+p((x_2,x_1))}$.
	
	Now, we suppose that the inequalities hold and consider a path $W \in \W$ with a corresponding compound ticket $(W_1,W_2)$.
	If $l_2(W_1) =0$ (w.l.o.g.), then $l_2(W)= l_2(W_2)=a$ and ${P(a)\leq P(0)+P(a)}$.
	If $l_2(W)=0$, then $d\defeq l_2(W_1) = l_2(W_2)\in \R_{\geq 0}$ and therefore $p(W)=P(0)\leq 2P(d) = p(W_1)+p(W_2)$.
	Now let all distances be strictly positive.
	We then have $l_2(W)\leq l_2(W_1)+l_2(W_2)$ and $p(W)=P(l_2(W)) \leq P(l_2(W_1))+ P(l_2(W_2))=p(W_1)+p(W_2)$.
\end{proof}

However, the no-elongation property is \emph{not} satisfied for the beeline tariff in general as the following small example demonstrates.

\begin{example}
	\begin{figure}[t]
		\centering
		\includegraphics[width=0.18\textwidth]{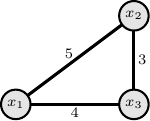}
		\caption{PTN in which the no-elongation property is not satisfied for the beeline tariff.}
		\label{Figure Luftlinie Lemma Eig2}
	\end{figure}
	Consider a beeline tariff with an affine price function $P(k)=f+\bar{p}\cdot k$ with $\bar{p}>0$ regarding the PTN depicted in Figure~\ref{Figure Luftlinie Lemma Eig2} and the path $W=(x_1,x_2)$. 
	The standard ticket $T=(W)$ costs $f+5 \cdot\bar{p}$, which is more than the price $f + 4 \cdot \bar{p}$ of the elongated ticket $T'=((x_1,x_2,x_3))$.
	Hence, the no-elongation property is not fulfilled, even if going back to the start station is not allowed.
\end{example}

In our example, passengers would save money by buying the elongated ticket $T'\,$, but leaving the bus already at station $x_2$.
This is avoided in practice, since passengers are tracked by their mobile devices and hence need to checkout at a station which is really visited.\medskip

We remark that instead of the Euclidean distance also other metrics can be used, see \cite{Urban20}.

\section{Zone-based Fare Structures} \label{section zone}
\emph{Zone tariffs} combine the ideas of flat and distance tariffs. 
The whole region is divided into \emph{tariff zones} and the length of a journey is approximated by the number of traversed zones.
In a counting zones pricing (which we consider), for all journeys traversing the same number of zones, the same price is charged.
A flat tariff is applied within a zone.
We start by analyzing the basic zone tariff and then extend our analysis to some particularities: metropolitan zones, overlap areas and single counting of zones (even if they are traversed multiple times).
The modeling is a little different for each of these particularities.

\subsection{Basic Zone Tariff} \label{section basic zone}

In order to model a basic zone tariff, we use the PTN. 
The geographical zones imply a zone partition $\cZ=\{Z_1,\ldots,Z_K\}$ of the set of stations $V\!$, i.e., $V=\bigcup_{i\in\{1,\ldots,K\}} Z_i$ and the $Z_i$ are pairwise disjoint.
It can happen that an edge $e$ between two stations crosses a zone without having a station in it.
We call such a situation an \emph{empty zone on edge $e$}. 
In this case, we add a virtual node which is not an actual station on edge $e$.
In other words, we assume that the PTN $(V,E)$ does not have empty zones on the edges, but some of the
nodes in $V$ might be virtual nodes.
For each node $v \in V\!$, let $\zones(v)\subseteq \cZ$ denote the zones to which node~$v$ belongs. 
In the basic zone tariff, $\zones(v)$ is a singleton, i.e., $\vert \zones(v) \vert = 1$.
This zone assignment yields the \emph{zone border weight}
\[b(e)=b(x,y) \defeq \begin{cases} 0 &\text{if } x \text{ and } y \text{ are in the same zone, i.e., } \zones(x) = \zones(y),\\ 1 & \text{ otherwise}\end{cases}\]
for all edges $e=\{x,y\} \in E$.
From that, we can derive for a path $W \in \W$
the \emph{zone function}
\begin{equation}	
	z(W) \defeq 1 + b(W), \mbox{ where } b(W) \defeq \sum_{e\in E(W)} b(e). \label{eq-z}
\end{equation}
This determines the number of zones which are traversed by the path $W$, where we count a zone multiple times, once for each time that it is traversed.
We illustrate the way to count zones in Example~\ref{zone Bsp zone system}.

\begin{example}\label{zone Bsp zone system}
	\begin{figure}[tbp]
		\centering
		\begin{subfigure}[c]{0.3\textwidth}
			\centering
			\includegraphics[scale=1]{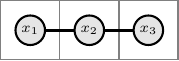}
			\caption{}
			\label{Subfig: 002}
		\end{subfigure}
		\begin{subfigure}[c]{0.3\textwidth}
			\centering
			\includegraphics[scale=1]{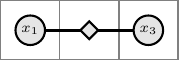}
			\caption{}
			\label{Subfig: 003}
		\end{subfigure}
		\begin{subfigure}[c]{0.3\textwidth}
			\centering
			\includegraphics[scale=1]{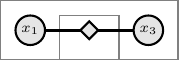}
			\caption{}
			\label{Subfig: 004}
		\end{subfigure}
		\caption{PTNs with zones for Example~\ref{zone Bsp zone system}.}
		\label{Figure Zonen Def}
	\end{figure}
	
	Figure~\ref{Figure Zonen Def} shows three different PTNs with zone partitions. 
	In Figures~\ref{Subfig: 003} and \ref{Subfig: 004}, there is an empty zone on edge $\{x_1,x_3\}$.
	Hence, we add a virtual node represented by the diamond-shaped node.
	In all three PTNs, the respective path $W$ from $x_1$ to $x_3$ crosses two zone borders, hence $z(W)=3$.
	Note that this is even true for the PTN in Figure~\ref{Subfig: 004} although ${\zones(x_1)=\zones(x_3)}$.
\end{example}

Another approach to model empty zones on edges instead of virtual nodes are zone border weights larger than~1.
This approach can be found in \cite{SchUrb-Atmos, Urban20}. 

The results that we derive in the following are the same for both modeling approaches. We give sufficient conditions for a fare structure to satisfy the no-stopover and no-elongation property. 
However, for a given basic zone tariff, it may make a difference to analyze the original PTN which may include empty zones or the one with virtual nodes for the no-stopover and no-elongation property if it is not forbidden to choose a virtual node as start or end node or to split a ticket there.\medskip

We now have the preliminaries to define a basic zone tariff.

\begin{definition}
	Let a PTN together with a zone partition $\cZ$ be given, and let $\W$ be the set of all paths in the PTN.
	A fare structure $p$ is a \emph{basic zone tariff} w.r.t.\ a price function $P \colon \N_{\geq 1} \rightarrow \R_{\geq 0}$ if $p(W)=P(z(W))$ for all $W \in \W$ where the zone function $z$ is defined as in~\eqref{eq-z}.
\end{definition}

The price function $P$ assigns a ticket price to every number of traversed zones.
If it is constant, the basic zone tariff simplifies to a flat tariff. 
A basic zone tariff can be interpreted as a distance tariff that measures the length of the path by the number of traversed zones.
As we will see, the properties of a basic zone tariff depend crucially on the price function $P$ and the results are similar to those for distance tariffs.

\begin{theorem}\label{zone1}
	Let a price function $P$ be given.
	All basic zone tariffs w.r.t.\ $P$ satisfy the no-stopover property if and only if
	\begin{equation}\label{eq-nostopzone}
		P(k) \leq P(i) + P(k-i+1) \mbox{ for all } k \in \N_{\geq 1},\, i \in \{1,\ldots, k\}.
	\end{equation}
	In particular, if all basic zone tariffs w.r.t.\ $P$ satisfy the no-stopover property, then the increase of the price function is bounded by $P(k) \leq (k-1)P(2)$ for $k \geq 2$.
\end{theorem}
\begin{proof}
	\begin{figure}
		\centering
		\includegraphics[scale=1]{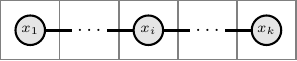}
		\caption{PTN with zones for Theorem~\ref{zone1}.}
		\label{Figure zone Thm Eig1}
	\end{figure}
	First, we assume that \eqref{eq-nostopzone} does not hold for some $k$ and $i$. 
	Note that $i \notin \{1,k\}$ since ${P(k)\leq P(1)+P(k)}$ is always true.
	We consider the basic zone tariff w.r.t.\ $P$ for the PTN depicted in Figure~\ref{Figure zone Thm Eig1}.
	In the induced basic zone tariff, the standard ticket of the path $(x_1,\ldots,x_k)$ costs~$P(k)$, whereas the compound ticket $([x_1,x_i],[x_i,x_k])$ costs $P(i)+P(k-i+1)$, which is cheaper by assumption.
	Hence, the no-stopover property is not satisfied.	
	
	Now, we suppose that \eqref{eq-nostopzone} holds.
	Let $p$ be any basic zone tariff w.r.t.~$P\!$.
	For a path ${W \in \W}$, we define $k \defeq z(W)$ and let $(W_1,W_2)$ be a corresponding compound ticket.
	It holds that ${z(W_1)+z(W_2)=k+1}$, i.e., ${z(W_2) = k-z(W_1)+1}$.
	By assumption it holds 
	\[p(W) = P(k) \leq P(z(W_1)) + P(k-z(W_1)+1) = p(W_1) + p(W_2).\]
	Thus, the no-stopover property is satisfied.
	\medskip
	
	For the second part of the theorem, let the no-stopover property be satisfied for all basic zone tariffs w.r.t.\ $P\!$. 
	Due to the first part of this proof, we know that \eqref{eq-nostopzone} holds, in particular for $i=2$, i.e., we have $P(k) \leq P(2) + P(k-1)$ for all $k\geq 2$.
	We prove the claim by induction over $k$.
	For $k=2$, the inequality is clearly fulfilled.
	For $k\geq 3$, we have
	\[P(k) \leq P(2) + P(k-1) \leq P(2) + (k-2)P(2) = (k-1)P(2).\]
	This proves the claim.
\end{proof}

Theorem~\ref{zone1} has an interesting interpretation, mathematically and from a transport point of view. 
Define the \emph{zone border price function} $\tilde{P} \colon \N_{\geq 1} \to \R_{\geq 0},\ k \mapsto P(k+1)$ as the function which maps the number of crossed zone borders to the price of the path. 
Then we have the following corollary.

\begin{corollary}
	All basic zone tariffs w.r.t.\ $P$ satisfy the no-stopover property if and only if
	the zone border price function $\tilde{P}$ is subadditive.
\end{corollary}
\begin{proof}
	The function $\tilde{P}$ is subadditive if and only if $\tilde{P}(i+k) \leq \tilde{P}(i) + \tilde{P}(k)$ for all $i,k \in \N_{\geq 1}$ which is the case if and only if $P(i+k-1) \leq P(i) + P(k)$ for all $i,k \in \N_{\geq 1}$ which in turn is equivalent to~\eqref{eq-nostopzone}.
\end{proof}

We also remark that \eqref{eq-nostopzone} holds if and only if the condition is satisfied for all $k \in \N_{\geq 3}$ and ${i \in \{2,\ldots, \lfloor \frac{k+1}{2} \rfloor \}}$, hence the number of cases to be checked can be decreased.

If the basic amount for using public transport is high enough in comparison to the additional price for traversing various zones, the no-stopover property is satisfied as the following lemma shows:

\begin{lemma}
	Let $p$ be a basic zone tariff w.r.t.\ a price function $P\!$.
	If for all $k \in \N_{\geq 1}$ it holds that $P(1)\leq P(k) \leq 2 P(1)$, i.e., in particular we have $P(1)\geq \frac{1}{2} \sup_{k\in \N_{\geq 1}} P(k)$, then the no-stopover property is satisfied.
\end{lemma}
\begin{proof}
	Consider a path $W\in \W$ with a corresponding compound ticket $(W_1,W_2)$.
	Then we have $p(W_1)+p(W_2)=P(z(W_1))+P(z(W_2))\geq 2 P(1) \geq P(z(W))=p(W)$, where the first inequality follows from the assumption that $P(k)\geq P(1)$ for all $k\in \N_{\geq 1}$.
\end{proof}

\begin{example} \label{zone Ex no-stopover}
	We provide some examples:
	\begin{itemize}
		\item In the unrealistic case that a price function $P$ is decreasing, every basic zone tariff w.r.t.\ $P$ satisfies the no-stopover property.
		\item For general increasing price functions, the no-stopover property need not be satisfied. 
		An example is a basic zone tariff in which a path passes through three consecutive zones and in which the zone prices are $P(1)=1, P(2)=2$ and $P(3)=5$.
		\item However, if the price function $P$ is affine and increasing, i.e., if $P(k) = f+ \bar{p} \cdot k$ with $\bar{p} \geq 0$ and~$f\geq - \overline{p}$, then every basic zone tariff w.r.t.\ $P$ satisfies the no-stopover property.
		This is a realistic choice of prices for a zone tariff.
	\end{itemize}
\end{example}

For the no-elongation property, there is the following criterion.

\begin{theorem}\label{zone-elongation}
	Let a price function $P$ be given.
	All basic zone tariffs w.r.t.\ $P$ satisfy the no-elongation property if and only if $P$ is increasing.
\end{theorem}

\begin{proof}
	Let $p$ be a basic zone tariff with respect to an increasing price function $P\!$.
	Note that for a path ${W=(x_1,\ldots,x_n) \in \W}$, $n\geq 2$, we have that $z([x_1,x_{n-1}]) \leq z(W)$. 
	Since $P$ is increasing, we obtain that $p([x_1,x_{n-1}])=P(z([x_1,x_{n-1}])) \leq  P(z(W))=p(W)$.
	
	If $P$ is not increasing, there is some $k \in \N_{\geq 2}$ such that $P(k)<P(k-1)$.
	Consider a basic zone tariff w.r.t.~$P$ in which there is a path $(x_1,\ldots,x_k)$ with $z([x_1,x_k])=k$ and $z([x_1,x_{k-1}])=k-1$.
	Then we have $p([x_1,x_k]) = P(k) < P(k-1) = p([x_1,x_{k-1}])$, and the no-elongation property is not satisfied in this basic zone tariff.
\end{proof}

We now turn our attention to cheapest (standard) tickets. 
First, note that cheapest paths need not exist in the case of a price function which contains a strictly decreasing sequence of zone prices $P(i_k)>P(i_{k+1})$, $(i_k)_{k \in N} \subseteq \N_{\geq 1}$, see \cite{SchUrb-Atmos}.
If the price function $P(k)$ becomes constant for $k \geq K$, then cheapest paths always exist.
Still, there might be cheapest standard tickets with large detours compared to a shortest path. All these situations are avoided if the price function is increasing.
However, even then, a cheapest path need not be unique and there might be two cheapest paths traversing different numbers of zones.

\begin{lemma}\label{zone2}
	Let $p$ be a basic zone tariff with an increasing price function $P\!$, and let $W\!$ be a path between $x,y \in V\!$.
	\begin{itemize}
		\item If $W$ traverses a minimum number of zones, it is a cheapest path from $x$ to $y$.
		\item If $P$ is strictly increasing, then $W$ is a cheapest path from $x$ to $y$ if and only if $W$ traverses a minimum number of zones.
	\end{itemize} 
	In both cases, the corresponding cheapest standard ticket $T=(W)$ is also a cheapest ticket if $p$ satisfies the no-stopover property.
\end{lemma}
\begin{proof}
	The first part is clear due to the monotonicity of $P\!$. 
	For the second part, we have ${P(k) < P(k+1)}$ for all $k \geq 1$.
	Hence, an $x$-$y$-path is cheapest if and only if it traverses a minimum number of zones.
	Lastly, the price function is increasing, i.e., the no-elongation property holds by Theorem~\ref{zone-elongation}. 
	If also the no-stopover property is satisfied, then a cheapest standard ticket is a cheapest ticket by Theorem~\ref{Prelim-thm}.
\end{proof}

In the case of an increasing price function, we can hence compute a cheapest path by shortest path techniques in the PTN with zone border weights as described in Algorithm~\ref{Alg zone}. 
\begin{algorithm}
	\caption{Basic zone tariff: finding a cheapest path.}
	\label{Alg zone}
	\SetKwInOut{Input}{Input}
	\SetKwInOut{Output}{Output}
	
	\Input{PTN $(V,E)$, two stations $x,y \in V$}
	\Output{$x$-$y$-path $W$}
	Compute a shortest $x$-$y$-path $W$ in the PTN by applying a shortest path algorithm using the zone border weight $b(e)$ as edge weight for $e \in E$.\\
	\Return $W$
\end{algorithm}

\begin{corollary} \label{zone-cor-alg}
	Let $p$ be a basic zone tariff with an increasing price function.
	\begin{itemize}
		\item Algorithm~\ref{Alg zone} computes a cheapest path and hence a cheapest standard ticket in polynomial time.
		\item If \eqref{eq-nostopzone} holds, Algorithm~\ref{Alg zone} yields a cheapest ticket in polynomial time.
	\end{itemize}
\end{corollary}
\begin{proof}
	The claim follows from Theorems~\ref{zone1} and~\ref{zone2} and from the fact that shortest path algorithms, e.g., Dijkstra, run in polynomial time.
\end{proof}

We remark that the number of nodes and edges in the graph in Algorithm~\ref{Alg zone} can be decreased by contracting edges with zone border weight $b(e)=0$ yielding the so-called \emph{zone graph}, see \cite{HaSc01,Urban20}.

\subsection{Zone Tariff with Metropolitan Zone}
\label{sec-metropolitan}
Many zone tariffs include particularities. 
A common one is the definition of metropolitan zones in which a subset of zones $\cZ_M \subseteq \cZ$ is combined to a common zone $Z_M = \bigcup_{Z \in \cZ_M} Z$, the \emph{metropolitan zone}. 
For journeys which cross the metropolitan zone or start or end there, the zones are counted as in the basic zone tariff.
For journeys within the metropolitan zone, a special price is fixed. 
A higher price might be charged if the metropolitan zone has a well-developed public transport network or is much larger than a usual zone.
A lower price might be chosen in order to make public transport more attractive, e.g., in city regions to reduce the car traffic.

Again, we assume that there are no empty zones on edges, otherwise we add virtual nodes in the same way as in Section~\ref{section basic zone}.
We say that a path $W=(x_1,\ldots,x_n)\in \W $ is \emph{included in the metropolitan zone $Z_M$} if $x_i \in Z_M$ for all $i\in \{1,\ldots,n\}$.

The formal definition of this fare structure is as follows:

\begin{definition}
	Let a PTN with a zone partition $\cZ$ and a metropolitan zone $Z_M$ be given, and let $\W$ be the set of all paths in the PTN.
	A fare structure~$p$ is a \emph{zone tariff with metropolitan zone}~$Z_M$, a price function $P \colon \N_{\geq 1} \rightarrow \R_{\geq 0}$ and a metropolitan price $P_M \in \R_{\geq 0}$  if we have for every path $W \in \W$ that
	\[ p(W) = \begin{cases}
	P_M \qquad & \text{if $W$ is included in the metropolitan zone $Z_M$,}\\
	P(z(W)) \qquad & \text{otherwise (where the zone function $z$ is defined as in~\eqref{eq-z})}.
	\end{cases}\]
\end{definition}

Note that zone tariffs with several metropolitan zones are also possible (and can be defined as above). 
Paths traveling through a metropolitan zone may also be treated in other ways, e.g., the metropolitan zone always counts as two zones, see \cite{Urban20}.

In order to simplify our analysis, we make the following assumptions:
\begin{itemize}
	\item the price function  $P \colon \N_{\geq 1} \rightarrow \R_{\geq 0}$ is increasing,
	\item the underlying basic zone tariff satisfies the no-stopover property,
	\item the metropolitan zone contains at least one node and is connected.
\end{itemize}
The third assumption comes without loss of generality because a disconnected metropolitan zone can be split into its connected components to obtain the same setting just with several connected metropolitan zones instead of one.

We first provide an example that the no-stopover property need not be satisfied for zone tariffs with metropolitan zones although it is satisfied for the underlying basic zone tariffs.

\begin{example}
	\label{ZM2 Ex Eig1}
	\begin{figure}[tbp]
	\centering
	\includegraphics[scale=1]{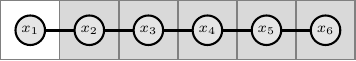}
	\caption{PTN with zones for Example~\ref{ZM2 Ex Eig1}.}
	\label{Figure ZM2 Ex Eig1}	
	\end{figure}
	Consider the PTN depicted in Figure~\ref{Figure ZM2 Ex Eig1}.
	The zones highlighted in gray form a metropolitan zone.
	Let $P \colon \N_{\geq 1} \rightarrow \R_{\geq 0},\ k \mapsto k$ be a linear price function, i.e., it satisfies the no-stopover property, see Example~\ref{zone Ex no-stopover}.
	Let $P_M \defeq 2 = P(2)$.
	The path $(x_1,x_2,x_3,x_4,x_5,x_6)$ costs $p([x_1,x_6])=P(6)=6$ with a standard ticket, but only ${p([x_1,x_2])+p([x_2,x_6]) = P(2) + P_M =4}$ with the compound ticket $((x_1,x_2),(x_2,\ldots,x_6))$, which benefits from the metropolitan zone.
\end{example}

Note that the situation described above occurs in the real world, e.g., in the fare structure of Verkehrsverbund Rhein-Neckar, a German public transport operator, see \cite{Urban20}.
In order to analyze in which cases the no-stopover property nevertheless holds, we define the \emph{maximum metropolitan zone distance} $\dmax$ by                                     
\[ \dmax \defeq \max_{x,y \in Z_M} \ \ \min_{x\text{-}y\text{-}\text{paths } W \text{ included in } Z_M} z(W). \]
The value $\dmax$ is the maximum number of zones that a shortest path which is completely contained in the metropolitan zone may traverse. It depends on the PTN (including the metropolitan zone) and is always finite due to the assumption that $Z_M$ contains at least one node and is connected. 
We remark that a shortest path may traverse the same zone twice, hence $\dmax$ can be larger than the number of zones belonging to the metropolitan zone. 
This is illustrated in Figure~\ref{Figure ZM2 Bem Eig1 2}, where three zones belong to a metropolitan zone.
A shortest path from $x$ to $y$ passes through the upper right zone twice, hence it passes four zones and we obtain that $\dmax=4>3$. 
Further, we assume that every passenger who travels within the metropolitan zone $Z_M$ uses a path with a minimum number of zones.
This yields $z(W)\leq \dmax$ for every path $W$ included in $Z_M$.
With this notation we state the following result:

\begin{figure}[tbp]
	\centering  
	\includegraphics[scale=1]{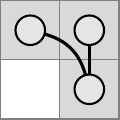}	
	\caption{In this PTN, $\dmax$ is larger than the number of zones that belong to the metropolitan zone.}
	\label{Figure ZM2 Bem Eig1 2}
\end{figure}

\begin{theorem}\label{zone3}
	Let an increasing price function $P\!$, a metropolitan price $P_M$, and an integer ${d\in \N_{\geq 1}}$ be given.
	All zone tariffs with metropolitan zone w.r.t.\ $P$ and $P_M$ on a PTN with $D_{\max} = d$ satisfy the no-stopover property if and only if 	$ P(d+k) \leq P_M + P(k+1)$ for all $k \in  \N_{\geq 1}$.
\end{theorem}
\begin{proof}
	First, assume that there is some $k$ such that $P(d +k) > P_M + P(k+1)$.
	Consider the PTN depicted in Figure~\ref{Figure ZM2 Lemma 2 Eig1}, where $\dmax=d$.
	In the induced zone tariff with metropolitan zone, the standard ticket of the path $(x_1,\ldots,x_{d+k})$, which costs $P(d+k)$, is more expensive than the compound ticket $([x_1,x_d],[x_d,x_{d+k}])$, which costs $P_M+P(k+1)$.
	\begin{figure}[t]
		\centering
		\includegraphics[scale=1]{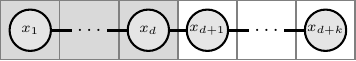}
		\caption{PTN with zones for Theorem~\ref{zone3}.}
		\label{Figure ZM2 Lemma 2 Eig1}	
	\end{figure}
	
	Conversely, we suppose that the inequalities are satisfied.
	Let a PTN with $\dmax= d$ be given. We show that the induced zone tariff with metropolitan zone satisfies the no-stopover property.
	By our assumptions, the no-stopover property is fulfilled for the basic zone tariff.
	Furthermore, it is satisfied for paths included in $Z_M$.
	Hence, we consider paths $W \in \W$ which are not included in $Z_M$, but allow to apply the metropolitan price for a subpath by making one stopover.
	Such a path must start or end in $Z_M$. 
	Let $W$ consist of the subpaths $W_1$ and $W_2$ where $W_1$ is included in $Z_M$ without loss of generality.
	We have $z(W_1)\leq \dmax=d$.
	Hence, it holds
	\begin{equation*}
	\begin{aligned}
		p(W) 	&= P(z(W)) = P(z(W_1) + z(W_2) -1) \stackrel{P \text{ incr.}}{\leq} P(d + z(W_2)-1)\\ 
		&\leq P_M + P(z(W_2)) = p(W_1)+p(W_2)
	\end{aligned}
	\end{equation*}
	and the no-stopover property is satisfied.
\end{proof}

\begin{theorem}\label{theo-elon-metro} 
	Let an increasing price function $P$ and a price $P_M$ be given.
	Then all zone tariffs with metropolitan zone w.r.t.\ $P$ and $P_M$ satisfy the no-elongation property if and only if $P_M \leq P(2)$.
\end{theorem}
\begin{proof}
	Let $p$ be a zone tariff with metropolitan zone $Z_M$ w.r.t.\ $P$ and $P_M$, and let ${W = (x_1,\ldots,x_n) \in \W}$, $n\geq 2$.
	We distinguish three cases.
	\begin{itemize}
		\item If $W$ is included in $Z_M$, then $p([x_1,x_{n-1}]) = P_M = p(W)$.
		\item If $[x_1,x_{n-1}]$ is included in $Z_M$, but $W$ is not, then $W$ traverses at least two zones and ${p([x_1,x_{n-1}]) = P_M \leq P(2) \leq p(W)}$ by assumption.
		On the other hand, if $P_M > P(2)$ and $W$ traverses exactly two zones, we obtain $p([x_1,x_{n-1}]) = P_M > P(2) = p(W)$ and the no-elongation property does not hold.
		\item If $[x_1,x_{n-1}]$ is not included in $Z_M$, then the prices of $W$ and its subpath~$[x_1,x_{n-1}]$ are computed as in the basic zone tariff. Hence, we have $p([x_1,x_{n-1}])\leq p(W)$ by monotonicity of $P$ and Theorem~\ref{zone-elongation}. \qedhere
	\end{itemize}
\end{proof}

We can make use of the following lemma to find a cheapest path. 

\begin{lemma}\label{ZM2 Lemma Eig3}
	Let $p$ be a zone tariff with metropolitan zone with an increasing price function $P$ and $P_M \leq P(3)$, and let $W$ be a path between $x,y \in V\!$.
	\begin{itemize}
		\item If $W$ is included in $P_M$, then it is a cheapest path from $x$ to $y$.
		\item If there does not exist an $x$-$y$-path included in $Z_M$, then $W$ is a cheapest path from $x$ to $y$ if~$W$ traverses a minimum number of zones.
		\item If $P$ is strictly increasing and there does not exist an $x$-$y$-path included in $Z_M$, then $W$ is a cheapest path from $x$ to $y$ if and only if $W$ traverses a minimum number of zones.
	\end{itemize}
	In all cases, the corresponding cheapest standard ticket $T=(W)$ is a cheapest ticket if ${P_M \leq P(2)}$ and $p$ satisfies the no-stopover property.
\end{lemma}

\begin{proof}
	Consider the first case in which $W$ is an $x$-$y$-path which is included in the metropolitan zone.
	It costs $p(W) = P_M$. 
	Any $x$-$y$-path that leaves the metropolitan zone traverses at least three zones and thus costs at least $P(3)\geq P_M$.
	Hence, $W$ is a cheapest path.
	If there is no path within the metropolitan zone, the price of a path is computed as in the basic zone tariff and Theorem~\ref{zone2} can be applied.
	
	If $P_M \leq P(2)$, then the no-elongation property is satisfied by Theorem~\ref{theo-elon-metro}.
	If also the no-stopover property holds, then a cheapest standard ticket is a cheapest ticket by Theorem~\ref{Prelim-thm}. 
\end{proof}

We conclude that we can compute a cheapest path in polynomial time in this case by Algorithm~\ref{Alg ZM2}.

\begin{algorithm}
	\SetKwInOut{Input}{Input}
	\SetKwInOut{Output}{Output}
	
	\Input{PTN $(V,E)$, two stations $x,y \in V$}
	\Output{$x$-$y$-path $W$}
	For each edge $e=\{v,w\}\in E$, define the metropolitan weight $z_M(e)$ by
	\[z_M(e) \defeq z_M(v,w) \defeq \begin{cases}
	0 &\text{if } v,w\in Z_M,\\
	1 &\text{otherwise,}
	\end{cases}\]
	and compute a shortest $x$-$y$-path in the PTN with $z_M(e)$ as edge weights.\\
	\eIf{$\sum_{e\in E(W)} z_M(e) =0$}{
	\Return $W$}{
	Apply Algorithm~\ref{Alg zone} for finding a cheapest path regarding the basic zone tariff.}
	\caption{Zone tariff with metropolitan zone: finding a cheapest path.}
	\label{Alg ZM2}
\end{algorithm}

\begin{corollary}
	Let $p$ be a zone tariff with metropolitan zone with an increasing price function~$P\!$.
	\begin{itemize}
		\item If $P_M \leq P(3)$, Algorithm~\ref{Alg ZM2} computes a cheapest path and hence a cheapest standard ticket in polynomial time.
		\item If $P_M \leq P(2)$ and $p$ satisfies the no-stopover property, then Algorithm~\ref{Alg ZM2} yields a cheapest ticket in polynomial time.
	\end{itemize}
\end{corollary}
\begin{proof}
	The metropolitan weight $z_M$ is chosen such that a path $W$ is included in the metropolitan zone if and only if $\sum_{e\in E(W)} z_M(e) =0$.
	Hence, the correctness follows from Lemma~\ref{ZM2 Lemma Eig3} and Corollary~\ref{zone-cor-alg}.
	The runtime is polynomial due to Corollary~\ref{zone-cor-alg} and because shortest path algorithms, e.g., Dijkstra, run in polynomial time.
\end{proof}

The case $P_M > P(3)$ remains: Here,
a cheapest $x$-$y$-path might leave the metropolitan zone
$Z_M$ although a path included in $Z_M$ exists, e.g., if there exists a path from $x$ to $y$ which leaves the metropolitan zone and traverses three zones in total.
Such a path can be determined in polynomial time by solving the single-source shortest path (SSSP) problem twice, namely from $x$ and from $y$ to all other nodes, then iterating over all edges leaving the metropolitan zone, i.e., edges between a node in $Z_M$ and a node not in $Z_M$, and complementing them with shortest paths to $x$ and $y$ in order to gain an $x$-$y$-path which leaves the metropolitan zone, and then choosing the smallest of these paths.
By comparing the price of this path with a path included in $Z_M$, we determine a cheapest path.
However, in this case, the no-elongation property does not hold and an elongated ticket might be cheaper than the cheapest standard ticket.
\medskip

We finally remark that due to Theorem~\ref{theo-elon-metro} already for $P_M > P(2)$ the no-elongation property is usually not satisfied. 
It is hence possible that a cheapest ticket between $x,y \in Z_M$ is an elongated ticket. 
Such an elongated ticket can also be found in polynomial time as follows:
One computes a path $W$ from $x$ to $y$ and outgoing paths from $x$ and from $y$ to all other zones, each minimizing the number of traversed zones.
Then $W$ is elongated by the shortest outgoing path.

\subsection{Zone Tariff with Overlap Areas} \label{sec-overlap}
Overlap areas, which allow stations to belong to several zones, are common in practice.
They are in particular used for stations near zone borders in order to make traveling in border regions more passenger-friendly.
This can be seen in Figure~\ref{Figure zoa Bsp Definition} where station $x_2$ is in an overlap area (depicted by the striped area) between the left and the right zone.
A traveler from $x_1$ to $x_2$ or a traveler from $x_2$ to $x_4$ traverses only one zone.
However, a traveler from $x_1$ via $x_2$ to $x_4$ traverses two zones.

In this section, we relax the requirement that the zones form a partition and allow a cover instead.
In particular, we may have stations (in overlap areas) which belong to more than one zone.
When determining the number of zones a path traverses, the zones for the stations in overlap areas are chosen in such a way that the total number of traversed zones of the path becomes minimal.
As before, we assume that there is at least one station in every zone (and in every overlap area).
In the following, we model this setting mathematically.\medskip

Given a zone cover $\cZ$, let $\zones(v) \subseteq \cZ$ with $\vert \zones(v)\vert \geq 1$ be the set of zones to which station~$v$ may be assigned. 
For a path $W=(x_1,\ldots,x_n)$, let $h \colon \{1,\ldots,n\} \to \cZ$ with $h(i)\in \zones(x_i)$ be the function that assigns a zone to each station in the path.
This assignment determines the zone border weights for any edge $\{x_i,x_{i+1}\}$, $i \in \{1,\ldots,n-1\}$, on the path $W$ as follows:
\[ b^h(x_i,x_{i+1}) \defeq \begin{cases}
0 &\text{if } h(i)=h(i+1),\\
1 &\text{otherwise.}
\end{cases}\] 
For the path $W\!$, we now choose an assignment $h$ which minimizes the number of zones that are traversed by $W\!$, i.e.,
\begin{equation}
	\label{eq-zneu}
	z(W):=1 + \min_{h}b^h(W) =
	1 + \min_{h} \sum_{i=1}^{n-1} b^h(x_i,x_{i+1}).
\end{equation}
The assignment of a station to a zone depends on the path which is shown in the
next example.

\begin{example}\label{zoa Bsp Definition}
	\begin{figure}[t]
		\centering
		\includegraphics[scale=1]{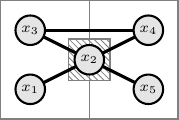}
		\caption{PTN with zones for Example~\ref{zoa Bsp Definition}.}
		\label{Figure zoa Bsp Definition}	
	\end{figure}
	Consider the PTN shown in Figure~\ref{Figure zoa Bsp Definition} with $\zones(x_1)=\zones(x_3)=\{L\}$, ${\zones(x_4)=\zones(x_5) =\{R\}}$ and $\zones(x_2)=\{L,R\}$. 
	The choice to which zone $x_2$ is assigned is made separately for every occurrence on every path.
	\begin{itemize}
		\item For $W=(x_1,x_2,x_3)$, we assign $h(2)=L$ and obtain $z(W) = 1+b^h(x_1,x_2) + b^h(x_2,x_3) = 1$.
		\item For $W=(x_4,x_2,x_5)$, we assign $h(2)=R$ and again obtain $z(W)=1$.
		\item For $W=(x_1,x_2,x_5)$, both possible assignments, i.e., $x_2$ in the left or in the right zone, yield $z(W)=2$.
		\item Finally, if a (non-simple) path traverses $x_2$ twice, we may assign $x_2$ even to two different zones as done for the path $W=(x_1,x_2,x_3,x_4,x_2,x_5)$, where we assign $x_2$ to the left zone for the first visit, i.e., $h(2)=L$, and to the right zone for the second visit, i.e., $h(5)=R$, to obtain 
		\begin{align*}
			z(W)&=1+b^h(x_1,x_2)+b^h(x_2,x_3)+b^h(x_3,x_4)+b^h(x_2,x_4)+b^h(x_2,x_5)\\
			&=1+0+0+1+0+0=2.
		\end{align*} 
	\end{itemize}
\end{example}

\begin{definition}\label{de-zoa}
	Let a PTN together with a zone cover $\cZ$ be given, and let $\W$ be the set of all paths in the PTN.
	A fare structure $p$ is a \emph{zone tariff with overlap areas (ZOA)} w.r.t.\ an increasing price function ${P \colon \N_{\geq 1} \rightarrow \R_{\geq 0}}$ if $p(W)=P(z(W))$ for all paths $W \in \W$ where $z$ is defined as in \eqref{eq-zneu}.
\end{definition}

This means that for each path $W$ a minimal assignment $h$ is fixed and the price for $W$ is then computed as a basic zone tariff.
Restricting a fixed minimal assignment of a path $W$ to a subpath does not necessarily yield a minimal assignment of the subpath.
Also a minimal assignment of a subpath of $W$ cannot necessarily be extended to a minimal assignment of $W\!$.
Furthermore, there not even need to be a minimal assignment for a path such that the restriction is also minimal for a subpath. 
Consider a line graph with four nodes $x_1, x_2, x_3, x_4$ and three zones $Z_1,Z_2,Z_3$ such that $Z_1=\{x_1,x_2\}$, $Z_2=\{x_2,x_3\}$ and $Z_3=\{x_3,x_4\}$, and let $P$ be strictly increasing. 
For the path $\{x_1,x_2,x_3,x_4\}$, the unique optimal solution is to assign $x_1$ and $x_2$ to $Z_1$ and $x_3$ and $x_4$ to $Z_3$.
However, for the path $\{x_2,x_3\}$ it is optimal to assign both nodes to $Z_2$.
Therefore, properties of the basic zone tariff cannot be transferred straightforwardly. 
Nevertheless, we always have ${z(W') \leq z(W)}$.

\begin{theorem}\label{zoa Thm Eig1}
	Let an increasing price function $P$ be given.
	All ZOAs w.r.t.\ $P$ satisfy the no-stopover property if and only if $P$ is subadditive.
\end{theorem}
\begin{proof}
	\begin{figure}[t]
		\centering
		\includegraphics[scale=1]{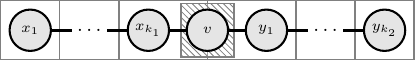}
		\caption{PTN with zones for Theorem~\ref{zoa Thm Eig1}.}
		\label{Figure zoa Lemma Eig1}	
	\end{figure}
	First assume that there are $k_1,k_2 \in \N_{\geq 1}$ such that $P(k_1+k_2)>P(k_1)+P(k_2)$. 
	We set ${k \defeq k_1+k_2}$ and consider the PTN depicted in Figure~\ref{Figure zoa Lemma Eig1}.
	The striped area is an overlap area so that node $v$ belongs to both neighboring zones.
	In the induced ZOA, the path ${W=(x_1,\ldots, x_{k_1},v,y_{1},\ldots, y_{k_2})}$ costs $p(W)=P(z(W))=P(k)$ no matter if we assign~$v$ to the left or to the right zone.
	Regarding the two subpaths $W_1 = (x_1,\ldots, x_{k_1},v)$ and $W_2 = (v,y_{1},\ldots, y_{k_2})$, we assign $v$ to the left zone to determine the number of zones for $W_1$ and to the right zone for $W_2$. 
	This means that the compound ticket~$(W_1, W_2)$ costs ${p(W_1)+p(W_2) = P(z(W_1))+P(z(W_2)) = P(k_1)+P(k_2)}$, which is cheaper than the standard ticket.
	Therefore, the no-stopover property is not satisfied.
	
	Conversely, we suppose that $P$ is subadditive.
	Let $p$ be a ZOA w.r.t.\ $P\!$, and let ${W=(x_1,\ldots,x_n)\in \W}$ be a path with a corresponding compound ticket $(W_1,W_2)$, i.e., ${W_1=[x_1,x_i]}$, $W_2= [x_i,x_n]$.
	Let $h, h_1, h_2$ be corresponding minimal assignments.
	We start by showing that ${z(W)\in \{z(W_1)+z(W_2)-1,z(W_1)+z(W_2)\}}$.
	We define new assignments for $W_1$ and $W_2$ by $h_1'(k)\defeq h(k)$ for $k \in \{1,\ldots,i\}$ and $h_2'(k) \defeq h(k+i-1)$ for $k \in \{1,\ldots,n-i+1\}$.
	This yields
	\begin{align*}
		z(W_1)+z(W_2) -1 \leq 1+\sum_{k=1}^{i-1} b^{h_1'}(x_k,x_{k+1}) + \sum_{k=i}^{n-1} b^{h_2'}(x_k,x_{k+1}) 
		= 1+\sum_{k=1}^{n-1} b^h(x_k,x_{k+1})= z(W).
	\end{align*}
	On the other hand, we can define a new assignment for $W$ by $h'(k)\defeq	h_1(k)$ for $k\in \{1,\ldots,i\}$ and $h'(k)\defeq h_2(k-i+1)$ for $k\in \{i+1,\ldots,n\}$.
	Then
	\begin{align*}
		z(W) \leq \underbrace{1+\sum_{k=1}^{i-1} b^{h'}(x_k,x_{k+1})}_{=z(W_1)} + \underbrace{b^{h'}(x_i,x_{i+1})}_{\leq 1} + \underbrace{\sum_{k=i+1}^{n-1} b^{h'}(x_k,x_{k+1})}_{\leq z(W_2)-1}
		\leq z(W_1)+z(W_2).
	\end{align*}
	In case that $z(W)=z(W_1) + z(W_2)$, it holds that
	\[ p(W) =P(z(W)) \stackrel{P \text{ subadd.}}{\leq} P(z(W_1)) + P(z(W_2)) = p(W_1)+p(W_2). \]
	Hence, consider the case that $z(W)=z(W_1) + z(W_2) -1$.
	It holds that
	\[ p(W) =P(z(W)) \stackrel{P \text{ incr.}}{\leq} P(z(W)+1) \stackrel{P \text{ subadd.}}{\leq} P(z(W_1)) + P(z(W_2))= p(W_1) + p(W_2). \]
	Therefore, the no-stopover property is satisfied.
\end{proof}

Analogously to Theorems~\ref{zone-elongation} and~\ref{zone2} for the basic zone tariff, we get the following results:
\begin{theorem}
	Let a price function $P$ be given.
	All ZOAs w.r.t.\ $P$ satisfy the no-elongation property if and only if $P$ is increasing.
\end{theorem}

\begin{lemma}\label{zoa-eig3}
	Let $p$ be a ZOA with an increasing price function $P\!$, and let $W$ be a path between $x,y \in V\!$.
	\begin{itemize}
	\item If $W$ traverses a minimum number of zones, i.e., it is an $x$-$y$-path which minimizes $z(W)$, it is a cheapest path from $x$ to $y$.
	\item If $P$ is strictly increasing, then $W$ is a cheapest path from $x$ to $y$ if and only if $W$ traverses a minimum number of zones.
	\end{itemize} 
	In both cases, the corresponding cheapest standard ticket $T=(W)$ is also a cheapest ticket if $p$ satisfies the no-stopover property.
\end{lemma}

In \eqref{eq-zneu} and hence in Definition~\ref{de-zoa} a \emph{minimal} assignment is needed.
Enumerating all possible assignments would lead to an exponential search. 
In the following, we show how such an assignment and even more a cheapest path for a ZOA can be computed in polynomial time.
To this end we construct a new graph, namely the overlaps-resolved graph, in which we resolve the overlap areas by creating new nodes and adding edge weights that represent the number of crossed zone borders.

\begin{definition}
	Let a PTN $(V,E)$ together with a zone cover $\cZ$ be given. 
	The \emph{overlaps-resolved graph} $G'=(V',E')$ is defined by      
	\begin{align*}
		V' & \defeq    V \cup V_{\cal Z} \ \mbox{ with } \ V_{\cal Z} \defeq \{(x,Z) \in V \times \cZ: Z \in \zones(x)\}, \\
		E' & \defeq \underbrace{ \{ \{x,(x,Z)\}: x \in V, (x,Z) \in V_{\cal Z} \}}_{\defeq E'_1} 
		\cup \underbrace{ \{ \{ (x,Z_1),(y,Z_2)\} \in V_{\cal Z} \times V_{\cal Z}:
			\{x,y\} \in E\}}_{\defeq E'_2}.
	\end{align*}       
	with weights
	\[ b'(e) \defeq \begin{cases}
	1			&\text{if } e \in E'_1,\\
	0			&\text{if } e=\{(x,Z_1),(y,Z_2)\}\in E'_2 \text{ and } Z_1=Z_2\\
	1			&\text{if } e=\{(x,Z_1),(y,Z_2)\}\in E'_2 \text{ and } Z_1 \neq Z_2.
	\end{cases} \]
	Note that for $e=\{(x,Z_1),(y,Z_2)\}$ and an assignment $h$ with $h(x)=Z_1, h(y)=Z_2$, we have ${b'(e)=b^h(x,y)}$.
\end{definition}

With $k_v:=|\zones(v)|$ for $v \in V\!$, the number of nodes and edges in $G'$ is ${\vert V' \vert = \vert V \vert + \sum_{v\in V} k_v}$ and $\vert E' \vert = \sum_{v\in V} k_v + \sum_{(v,w) \in E} k_v \cdot k_w$.
This means that they increase linearly and quadratically in the maximum number $\kmax \defeq \max_{v \in V} k_v$ of zones to which one station $v$ may belong: ${\vert V' \vert \in \mathcal{O}\bigl(\vert V \vert \cdot (\kmax+1) \bigr)}$ and ${\vert E' \vert \in \mathcal{O}(\kmax \cdot \vert V \vert + \kmax^2 \vert E \vert)}$.
This is polynomial in the input $(V, E, \cZ)$ since $\kmax\leq \vert \cZ \vert$.
In practice, $\kmax$ is usually small, often even $\kmax\leq 2$.
Note that the number of nodes and edges can be decreased by removing the node $v$ if $k_v=1$ which in practice happens in most cases.

\begin{example} \label{zoa Bsp overlaps-resolved graph}	
We have a look at the construction of an overlaps-resolved graph in the example depicted in Figure~\ref{Figure zoa Bsp overlaps-resolved graph}.
In \ref{Subfig: 037} we have a PTN with $\zones(x)=\{L\}$, ${\zones(v)=\{L,R\}}$ and ${\zones(y)={R}}$.
In \ref{Subfig: 038} the corresponding overlaps-resolved graph is shown.
The edges in $E'_1$ are dashed and have weight 1. 
Since $v$ belongs to two zones, it is represented by two nodes $(v,L), (v,R)$ in the overlaps-resolved graph, where $(v,L)$ represents~$v$ belonging to the left zone and $(v,R)$ to the right zone.
The node $v$ in $G'$ allows us to start or end a path in~$v$ without considering each node $(v,Z)$, $Z\in \zones(v)$, as the start or end node or determining to which zone $v$ should belong beforehand.
The misuse of edges in $E'_1$ is prevented by their weight.
\begin{figure}[tbp]
	\centering
	\begin{subfigure}[b]{0.4\textwidth}
		\centering
		\includegraphics[scale=1]{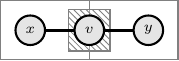}
		\caption{}
		\label{Subfig: 037}
	\end{subfigure}
	\begin{subfigure}[b]{0.4\textwidth}
		\centering
		\includegraphics[scale=1]{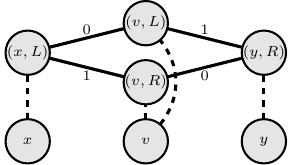}
		\caption{}
		\label{Subfig: 038}
	\end{subfigure}
	\caption{Construction of an overlaps-resolved graph for Example~\ref{zoa Bsp overlaps-resolved graph}.}
	\label{Figure zoa Bsp overlaps-resolved graph}
	\end{figure}
\end{example}

Now, we can apply any shortest path algorithm to the overlaps-resolved graph $G'$ corresponding to the PTN in order to compute paths that traverse a minimum number of zones for a ZOA as described in Algorithm~\ref{Alg zoa}.
If the path is already given and we are only interested in the minimal assignment, we can also use Algorithm~\ref{Alg zoa} where we reduce the PTN (and hence the overlaps-resolved graph) to the given path~$W\!$.
\begin{algorithm}
	\caption{ZOA: finding a cheapest path.}
	\label{Alg zoa}
	\SetKwInOut{Input}{Input}
	\SetKwInOut{Output}{Output}
	
	\Input{overlaps-resolved graph $G' = (V',E')$, two stations $x,y \in V$}
	\Output{$x$-$y$-path $W'$ in $G'$}
	
	Compute a shortest $x$-$y$-path $W'$ in $G'$.\\
	\Return $W'$
\end{algorithm}

\begin{lemma}
	Let $p$ be a ZOA with an increasing price function. Algorithm~\ref{Alg zoa} yields a cheapest path~$W$ from $x$ to $y$ in polynomial time.
\end{lemma}
\begin{proof}
	Let $W'=(x,v'_1,v'_2,\ldots,v'_K,y)$ be the output of Algorithm~\ref{Alg zoa}.
	We first note that $v'_i \in V_{\cal Z}$ for all $i \in \{1,\ldots,K\}$: 
	Assume $v \defeq v'_i \notin V_{\cal Z}$ for some $i$, so we have $v \in V \setminus \{x,y\}$. 
	Since the node $v$ is only incident to edges in $E'_1$, this means that $W'$ contains a subpath $W''\defeq ((w,Z_1),(v,Z_2),v,(v,Z_3))$ with $\{w,v\} \in E'_2$. 
	Hence, also $e':=\{(w,Z_1),(v,Z_3)\} \in E'_2$. 
	Replacing the subpath $W''$ in $W'$ by $e'$ leads to a new path in $G'$ with strictly smaller weight due to $b'(W'')\geq 2 > b'(e')$. This is a contradiction to $W'$ being a shortest path in $G'$. 
	
	We now transfer $W'$ to a path $W=(v_1,\ldots,v_K)$ in the PTN by replacing $v'=(v,Z)$ by the projection on its first component.
	This yields $v_1=x,\ v_K=y$ and $\{v_i,v_{i+1}\} \in E$ for all ${i \in \{1,\ldots,K-1\}}$ by definition of $E'_2$.
	We hence obtain an $x$-$y$-path in the PTN together with an assignment $h_1$ which maps $h_1(v)=Z$ if $v'=(v,Z) \in W'$. 
	Due to the definition of the edge weights in $G'$, we get $b'(W')=b^{h_1}(W)+2$ adding the weights for the first and last edge in $W'$, and $h_1$ is a minimal assignment for $W\!$.
	It remains to show that it is a cheapest path.
	
	Assume that $W^*=(w_1,\ldots,w_L)$ with $w_1=x$, $w_L=y$, is a cheaper $x$-$y$-path in the PTN with the minimal assignment~$h_2$, i.e., $b^{h_2}(W^*)=z(W^*)<z(W) = b^{h_1}(W)$. 
	We lift $W^*$ to the path 
	\[ {W^*}'=\big (w_1,(w_1,h_2(w_1)),(w_2,h_2(w_2)),\ldots,(w_L,h_2(w_L)),w_L \big ). \]
	Then $b'({W^*}')=b^{h_2}(W^*)+2 < b^{h_1}(W) + 2$, a contradiction to $W'$ being a shortest path in $G'$.
	
	As shown above, the overlaps-resolved graph is polynomial in the input and hence Algorithm~\ref{Alg zoa} can be solved in polynomial time due to the polynomial runtime of shortest path algorithms like Dijkstra.
\end{proof}

\begin{corollary}
	Let $p$ be a ZOA with an increasing price function $P\!$.
	\begin{itemize}
	\item Algorithm~\ref{Alg zoa} computes a cheapest path and hence a cheapest standard ticket in polynomial time.
	\item If $P$ is subadditive, then Algorithm~\ref{Alg zoa} yields a cheapest ticket in polynomial time.
	\end{itemize}
\end{corollary}

\subsection{Zone Tariff with Single Counting} \label{sec-single}
In this section, we study zone tariffs as they are usually implemented in practice.
In contrast to the basic zone tariff with multiple counting, each zone is only counted once independent of how many times it is entered on a path. 
More precisely, the price of a path is determined by the number of different zones traversed along a path.
Again, we assume that there are no empty zones on edges, otherwise we add virtual nodes in the same way as in Section~\ref{section basic zone}.

As for the basic zone tariff, the set $\zones(v) \subseteq Z$ with $\vert \zones(v) \vert = 1$ denotes the zone to which node $v \in V$ belongs and the set of resulting zones $\cZ$ is a partition of $V\!$.
However, in a zone tariff with single counting, we need a different \emph{zone function} which we define for every path $W \in \W$ by counting the number of \emph{different} zones which are traversed by $W\!$, i.e.,
\begin{equation}\label{eq-zbar}
	\bar{z}(W) := \vert \bigcup_{x\in V(W)} \zones(x) \vert.
\end{equation}

\begin{definition}
	Let a PTN together with a zone partition $\cZ$ be given.
	Let $\W$ be the set of all paths in the PTN.
	A fare structure $p$ is a \emph{zone tariff with single counting} w.r.t.\ a price function
	$P \colon \N_{\geq 1} \rightarrow \R_{\geq 0}$ if $p(W)=P(\bar{z}(W))$ for each path $W \in \W$ where the zone function $\bar{z}$ is defined as in~\eqref{eq-zbar}.
\end{definition}

For a path that does not traverse a zone more than once, the basic zone tariff and the zone tariff with single counting coincide. 
Hence, many examples and results from the basic zone tariff can be transferred to the zone tariff with single counting.
However, in contrast to the case for the basic zone tariff, cheapest paths always exist for a zone tariff with single counting, even if the price function is decreasing, as we show in the following lemma.

\begin{lemma} \label{zone-no-double-cheapest-path}
	For a zone tariff with single counting, there always exist cheapest paths.
\end{lemma}
\begin{proof}
	Let $x,y \in V$ be stations in the PTN.
	Since we do not count zones multiple times, the number of traversed zones of an $x$-$y$-path is bounded from above by the total number of zones $\vert \cZ \vert \in \N_{\geq 1}$ in the PTN.
	Let~$M$ be the set of all prices that are possible for $x$-$y$-paths.
	We have $M \subseteq \{ P(1),\ldots, P(\vert \cZ \vert) \}$.
	Hence, $\vert M \vert < \infty$ and $M$ admits a minimum.
	Therefore, there exists a cheapest $x$-$y$-path.
\end{proof}

Due to practical relevance, we again focus on increasing price functions.

\begin{theorem}\label{zone no double Thm Eig1 incr}
	Let an increasing  price function $P$ be given.
	All zone tariffs with single counting w.r.t.\ $P$ satisfy the no-stopover property if and only if \eqref{eq-nostopzone} holds, i.e., if and only if
	${P(k) \leq P(i) + P(k-i+1)} \mbox{ for all } k \in \N_{\geq 1},\, i \in \{1,\ldots, k\}$.
\end{theorem}
\begin{proof}
	If \eqref{eq-nostopzone} does not hold for some k and i, then the no-stopover property does not hold as shown in the proof of Theorem~\ref{zone1}.
	
	Conversely, let \eqref{eq-nostopzone} hold.
	Let $p$ be any zone tariff with single counting w.r.t.\ $P\!$. 
	For a path $W\in \W$, we define $k\defeq \bar{z}(W)$ and let $(W_1,W_2)$ be a corresponding compound ticket.
	Then $W_1$ and $W_2$ have at least one zone in common, namely the zone in which we make the stopover.
	This is only counted once for $W\!$.
	Therefore, we have $\bar{z}(W_1)+\bar{z}(W_2) \geq k+1$ which is equivalent to $k-\bar{z}(W_1)+1\leq \bar{z}(W_2)$.
	With $i \defeq \bar{z}(W_1)\leq k$, we conclude from \eqref{eq-nostopzone} that
	\[p(W)=P(k)\leq P(i)+P(k-i+1)\leq P(\bar{z}(W_1))+P(\bar{z}(W_2))=p(W_1)+p(W_2).\]
	So the no-stopover property is satisfied.
\end{proof}

Theorems~\ref{zone-elongation} and~\ref{zone2} about the no-elongation property and cheapest paths in the basic zone tariff are analogously true for the zone tariff with single counting, which yields the following results:

\begin{theorem}\label{zone no double Thm Eig2}
	Let a price function $P$ be given.
	All zone tariffs with single counting w.r.t.\ $P$ satisfy the no-elongation property if and only if $P$ is increasing.
\end{theorem}

\begin{lemma}\label{zone no double Cor Eig3}
	Let $p$ be a zone tariff with single counting with an increasing price function $P\!$, and let $W$ be a path between $x,y \in V\!$.
	\begin{itemize}
	\item If $W$ traverses a minimum number of different zones, it is a cheapest path from $x$ to $y$.
	\item If $P$ is strictly increasing, then $W$ is a cheapest path from $x$ to $y$ if and only if $W$ traverses a minimum number of different zones.
	\end{itemize} 
	In both cases, the corresponding cheapest standard ticket $T=(W)$ is also a cheapest ticket if $p$ satisfies the no-stopover property.
\end{lemma}

Due to Lemma~\ref{zone no double Cor Eig3}, we are interested in an algorithm that computes a path which traverses a minimum number of different zones.
Unfortunately, finding such a path is NP-hard.
Let us consider the decision version of this problem, which we will call \textsc{Minimum-Zone Path} (MZP).
It can be stated as follows:
\begin{description}
	\item[Instance:] PTN $(V,E)$ involving a zone partition for the zone tariff with single counting (without empty zones on edges), i.e., zone sets $\zones(v)$ for all $v \in V\!$, nodes $x,y \in V$ and an integer $K \in \N_{\geq 1}$.
	\item[Question:] Is there a path from $x$ to $y$ that traverses at most $K$ different zones?
\end{description}

MZP was used in a similar setting by~\cite{Blanco2016}.
They deal with the \textsc{Shortest Path Problem With Crossing Costs} (SPPCC) when optimizing flight trajectories with overflight costs. While we use a \emph{counting zones pricing}, the case of SPPCC with constant crossing cost functions (SPPCC/C/$\cdot$) is a \emph{cumulative pricing} (see the terminology of \cite{OttoBoysen17}). The special case with arc weights set to 0 and crossing costs of 1, which is considered in Proposition~3 of \cite{Blanco2016}, coincides with MZP and has already been shown to be NP-complete there.
We give the proof for our setting:

\begin{theorem}\label{zone no double Thm NP-complete}
	MZP is NP-complete.
\end{theorem}
\begin{proof}
	MZP is in NP, since given a path $W\!$, it can be checked in polynomial time whether it traverses at most $K$ different zones, i.e., if $\bar{z}(W) \leq K$. 
	
	For the polynomial reduction, we use \textsc{Minimum-Color Single-Path} (MCSiP), which was introduced and shown to be NP-complete by \cite{MinimumColorPath}.
	Its decision version can be stated as follows:
	\begin{description}
		\item[Instance:] Graph $G=(V,E)$, a finite set of colors $C$, a function $c \colon E \to C$ which assigns a color to each edge, nodes $x,y \in V$ and an integer $k \in \N$.
		\item[Question:] Is there a path from $x$ to $y$ that uses at most $k$ colors?
	\end{description}
	Let a graph $G =(V,E)$, a set of colors $C$ with the assignment function $c \colon E \to C$, two nodes $x,y \in V$ (w.l.o.g. we assume $x \not=y$) and an integer $k \in \N_{\geq 1}$ be given.
	We construct an instance for MZP.
	In order to obtain a PTN for MZP, we add a node on each edge to represent the color of the edge.
	Hence, we define the sets of nodes and edges by
	\begin{align*}	
		&V' \defeq V \cup \{v_e: e\in E\},
		&E' \defeq \{\{x,v_e\},\{v_e,y\}: \{x,v\}=e\in E\}.
	\end{align*}
	This is polynomial since $\vert V' \vert = \vert V\vert + \vert E \vert$ and $\vert E'\vert = 2 \vert E \vert$.
	Finally, we define the zone partition as follows by introducing a dummy zone with the label Null which is not in~$C$ (see Figure~\ref{Figure zone no double Thm NP-complete}):
	\[
	\zones(x) \defeq
	\begin{cases}
	\{\textup{Null}\} 	&\text{if } x \in V\\
	\{c(e)\}			&\text{if } x\in V'\setminus V.
	\end{cases}
	\]
	Thus, a path traverses at most $K \defeq k+1$ zones if and only if it uses at most $k$ colors.
	\begin{figure}[tbp]
		\centering
		\includegraphics[scale=1]{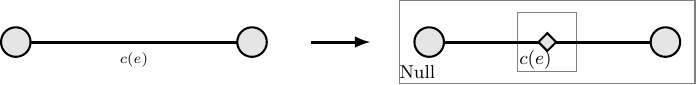}
		\caption{Construction of a zone partition in the proof of Theorem~\ref{zone no double Thm NP-complete}.}
		\label{Figure zone no double Thm NP-complete}	
	\end{figure}
\end{proof}

\begin{corollary}
	The cheapest ticket problem is NP-hard for a zone tariff with single counting, even if
	the price function is strictly increasing and the no-stopover property is satisfied.
\end{corollary}
\begin{proof}
	Solving the cheapest ticket problem with a strictly increasing price function that satisfies the no-stopover property is equivalent to finding a path which traverses a minimum number of zones (Lemma~\ref{zone no double Cor Eig3}). 
	Hence, the result follows from Theorem~\ref{zone no double Thm NP-complete}.
\end{proof}  

We now investigate special cases in which the problem can nevertheless be solved in polynomial time.

\begin{lemma} \label{zone-no-double-special-1}
	Let a PTN with a zone partition $\cal Z$ be given.
	If all zones $Z \in \cal Z$ are connected, then MZP can be solved in polynomial time.
\end{lemma}
\begin{proof}
	Let the path $W$ be a solution to MZP.
	Assume there exists a zone $Z$ which is traversed twice by $W\!$, i.e., $W$ contains a subpath $W'=(x,y_1,\ldots,y_k,x')$, $k\geq 1$, with $x,x' \in Z$ and $y_i \not\in Z$ for all $i \in \{1,\ldots,k\}$. 
	Since $Z$ is connected there exists a path $W''$ from $x$ to $x'$ which does not leave $Z$.
	Replacing $W'$ by $W''$ in $W$ hence traverses at most as many different zones as $W\!$.
	Therefore, we do not need to store the sets of traversed zones in this case, but it is enough to find a path crossing a minimal number of zone borders.
	This means that MZP can be solved in polynomial time by a shortest path algorithm in the PTN with the zone border weight $b$ from the basic zone tariff as edge weight (Algorithm~\ref{Alg zone}).
\end{proof}

In the following we generalize this result and allow even one zone which is not connected. More precisely,
we consider a PTN with a zone tariff with single counting with respect to an increasing price function~$P\!$. 
Let $b$ be the zone border weight as for the basic zone tariff.
We suppose that at most one zone $Z$ is disconnected and decomposes into $k$ components.
In this case, Algorithm~\ref{Alg no double 2} computes a cheapest path with single counting.

\begin{algorithm}
	\caption{Special case of a zone tariff with single counting: finding a cheapest path.}
	\label{Alg no double 2}
	\SetKwInOut{Input}{Input}
	\SetKwInOut{Output}{Output}
	
	\Input{PTN $(V,E)$ with a zone partition $\cZ$, i.e., zone border weights $b$, in which all zones but one zone $Z$ are connected and $Z$ decomposes into $k$ connected components, two stations $x,y \in V$} 
	\Output{$x$-$y$-path $W$}
	
	Create a directed graph $G = (V,A)$ by replacing each edge in the PTN by two directed arcs, one for each direction.
	For every arc $(v,w) \in A$, let the weight $\bar{b}(v,w)$ be given by 
	\[ \bar{b}(v,w)\defeq \begin{cases}
	\frac{1}{k} 	&\text{if }  w \in Z \text{ and } v \notin Z,\\
	b(v,w)	&\text{otherwise.}
	\end{cases}\]\\
	Compute a shortest $x$-$y$-path $W$ in $G$.\\
	\Return $W$
\end{algorithm}

\begin{lemma}
	Algorithm~\ref{Alg no double 2} yields a cheapest path and hence a cheapest standard ticket in polynomial time for the special case of an increasing price function and at most one disconnected zone. 
	
	In particular, if \eqref{eq-nostopzone} holds, then Algorithm~\ref{Alg no double 2} yields a cheapest ticket.
\end{lemma}
\begin{proof}
	Note that a path in $G$ can be understood as a path in the PTN and vice versa.
	Let~$W$ be a path from $x$ to $y$ computed by Algorithm~\ref{Alg no double 2}.
	Since $W$ is a shortest path regarding $\bar{b}$, it does not traverse a zone in $\cZ\setminus \{Z\}$ more than once due to connectedness as in the proof of Lemma~\ref{zone-no-double-special-1}.
	Therefore, if $W$ has weight $\bar{b}(W)=m$ in~$G$, it holds that
	\begin{equation}
	\bar{z}(W)=\begin{cases}
	\lceil m+1 \rceil &\text{if } x \notin Z,\\
	\lceil m+\frac{1}{k} \rceil &\text{if } x \in Z,
	\end{cases}  \label{eq-nodouble}
	\end{equation}
	because a shortest path $W$ traverses $Z$ at most $k$ times.
	From that, we can conclude that $W$ traverses a minimum number of different zones:
	Assume there is a path $W'$ which traverses fewer zones than $W\!$, i.e., $\bar{z}(W')<\bar{z}(W)$.
	Regarding the proof of Lemma~\ref{zone-no-double-special-1}, we can assume that $W'$ does not traverse a zone in $\cZ\setminus\{Z\}$ more than once, and hence \eqref{eq-nodouble} holds for~$W'$.
	Since $W$ is a shortest path regarding $\bar{b}$, we have $\bar{b}(W) \leq \bar{b}(W')$.
	Because $W$ and $W'$ have the same start node and due to \eqref{eq-nodouble} and monotonicity of the ceiling function, this is a contradiction to~$W'$ traversing less zones than $W\!$.
	
	The graph $G$ can be constructed in polynomial time and the runtime of shortest path algorithms, e.g., Dijkstra, is polynomial, hence the overall runtime is polynomial as well.
	
	If the price function is increasing and \eqref{eq-nostopzone} holds, then the no-stopover and no-elongation property are satisfied by Theorems~\ref{zone no double Thm Eig1 incr} and~\ref{zone no double Thm Eig2}.
	Hence, a cheapest standard ticket is a cheapest ticket by Theorem~\ref{Prelim-thm}.
\end{proof}

\section{Combined Fare Structures} \label{chapter combined}

Real-world public transport fares are often very complex, and usually not only one but several different fare structures are implemented within the same linked transport system.
Therefore, we have a look at combinations of two fare structures both being available in the same geographical region. In this case, passengers choose the cheapest available tickets.
We start with the definition of a combined fare structure and some general properties before we discuss two cases of combined fare structures: the bounded distance tariff and a combination of a zone tariff and a short-distance tariff.

\subsection{General Properties of Combined Fare Structures} \label{sec-combined}

\begin{definition}\label{combined Definition}
	Let a PTN be given, and let $p_1$ and $p_2$ be two fare structures.
	A \emph{combined fare structure} $p$ of $p_1$ and $p_2$ is defined by
	$ p(W) \defeq \min \{p_1(W),p_2(W) \} $
	for all $W \in \W$.
\end{definition}

One could think that if both fare structures satisfy a property, then so does their combined fare structure.
While this is indeed true for the no-elongation property, it is not true for the no-stopover property.

\begin{example}\label{combined Ex Eig1}
	The no-stopover property does not transfer to the combined fare structure. 
	To see this, we consider the PTN shown in Figure~\ref{Figure combined Lemma Eig1} with $l(x_1,x_2) = l(x_2,x_3)=2$, where we omit virtual nodes to simplify the presentation.
	Let $p_1$ be the basic zone tariff with respect to a linear price function ${P \colon \N_{\geq 1} \to \R_{\geq 0}},\ k \mapsto k$, and let $p_2$ be a distance tariff with an affine price function with $f =0$ and $\overline{p} = 1$, i.e., for $W\in \W$ we have $p_1(W) = z(W)$ and $p_2(W)=l(W)$.
	By Theorem~\ref{distance-eig1} and Example~\ref{zone Ex no-stopover}, both fare structures satisfy the no-stopover property.
	We define the paths ${W \defeq (x_1,x_2,x_3)}$, ${W_1 \defeq (x_1,x_2)}$ and $W_2 \defeq (x_2,x_3)$.
	The resulting costs are presented in Table~\ref{Table combined Lemma Eig1}.
	The no-stopover property is not satisfied for the combined fare structure because $p(W_1)+p(W_2) = 3 < 4 = p(W)$.
\end{example}

\begin{figure}[b]
	\centering
	\includegraphics[scale=1]{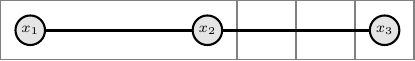}
	\caption{PTN with zones for Example~\ref{combined Ex Eig1}.}
	\label{Figure combined Lemma Eig1}	
\end{figure}

\begin{table}[t]
	\caption{Prices of the paths regarding the different fare structures from Example~\ref{combined Ex Eig1}.}
	\label{Table combined Lemma Eig1}
	\centering
	\begin{tabular}{|c|c|c|c|} \hline
		&	$W_1$	&	$W_2$	&	$W$	\\ \hline \hline
		$p_1$	&	1		&	4		&	4	\\ \hline	
		$p_2$	&	2		&	2		&	4	\\ \hline
		$p$		&	1		&	2		&	4	\\ \hline
	\end{tabular}
\end{table}

\begin{theorem}\label{combined Thm Eig2}
	Let $p_1$ and $p_2$ be two fare structures which satisfy the no-elongation property.
	Then their combined fare structure $p$ satisfies the no-elongation property as well.
\end{theorem}
\begin{proof}
	Let $W = (x_1,\ldots,x_n) \in \W$ with $n \geq 2$ be given and set $W_1 \defeq [x_1,x_{n-1}]$.
	Due to the no-elongation property for $p_1$ and $p_2$, it holds that
	\[ p(W_1) = \min \{p_1(W_1),p_2(W_1)\} \leq 
	\begin{cases}
		p_1(W_1) \leq p_1(W),\\
		p_2(W_1) \leq p_2(W).
	\end{cases}       \]
	Hence, it holds that
	\[ p(W_1) = \min \{p_1(W_1),p_2(W_1)\} \leq \min \{ p_1(W) , p_2(W) \} = p(W)\]
	and $p$ satisfies the no-elongation property.
\end{proof}

We end the section with the following simple, but algorithmically relevant observation:
\begin{equation} \label{eq-comb-cheapest}
	\min_{x\text{-}y\text{-path } W} \, p(W) = \min_{x\text{-}y\text{-path } W}\, \min_{p'\in \{p_1,p_2\}} p'(W) = \min_{p'\in \{p_1,p_2\}} \, \min_{x\text{-}y\text{-path } W} \, p'(W). 
\end{equation}

This means that we can find a cheapest path and hence a cheapest standard ticket by comparing the cheapest path w.r.t.\ $p_1$ and the cheapest path w.r.t.\ $p_2$ and choosing the better of the two. 
This is stated in Algorithm~\ref{algo-combined}.

\begin{algorithm}
	\caption{Combined fare structure: finding a cheapest path.}
	\label{algo-combined}
	\SetKwInOut{Input}{Input}
	\SetKwInOut{Output}{Output}
	\SetKwFunction{sdp}{SDP}
	\SetKw{AND}{and}
	
	\Input{PTN $(V,E)$, combined fare structure $p$ of $p_1$ and $p_2$, two stations $x, y \in V$}
	\Output{$x$-$y$-path $W$}
	
	Compute a cheapest path $W_1$ w.r.t.\ $p_1$.\\
	Compute a cheapest path $W_2$ w.r.t.\ $p_2$.\\
	\eIf{$p_1(W_1) \leq p_2(W_2)$}{
		\Return  $W_1$
	}{
		\Return $W_2$
	}
\end{algorithm}

\subsection{Bounded Distance Tariff} \label{sec-bounded}
The price of a journey can usually not become arbitrarily large, i.e., there is an upper bound on the price.
This can be modeled as a combined fare structure with the help of flat tariffs.
Consider, for example, a distance tariff with an affine price function $p_1(W)=f+\overline{p}_1 \cdot l(W)$ with $\overline{p}_1>0$ and a flat tariff $p_2(W)=\overline{p}_2$ for $W\in \W$.
The \emph{bounded distance tariff} as the combined fare structure of a distance tariff and a flat tariff is then given by 
\begin{equation*}
	p(W) = \min \{p_1(W),p_2(W)\}=\begin{cases}
		p_1(W) = f+\overline{p}_1\cdot l(W) 			& \text{if } l(W) \leq \frac{\overline{p}_2 -f}{\overline{p}_1},\\
		p_2(W) = \overline{p}_2							& \text{otherwise.}
	\end{cases}
\end{equation*}

This can be regarded as a distance tariff with the price function
\begin{equation*}
	P(l(W)) =\begin{cases}
		f+\overline{p}_1\cdot l(W) 			& \text{if } l(W) \leq \frac{\overline{p}_2 -f}{\overline{p}_1},\\
		\overline{p}_2						& \text{otherwise,}
	\end{cases}
\end{equation*}
which yields the following result:

\begin{theorem} \label{bounded-distance-Thm-Eig12}
	A bounded distance tariff satisfies the no-stopover and the no-elongation property.
\end{theorem}
\begin{proof}
	Since $P$ is a subadditive and increasing function, the no-stopover and no-elongation property are satisfied by Theorems~\refeq{distance-eig1} and~\ref{distance-eig2}.
	The no-elongation property can also be proved by Theorems~\ref{combined Thm Eig2} and~\ref{distance-eig2}.
\end{proof}

A cheapest standard ticket and thus by Theorems~\ref{bounded-distance-Thm-Eig12} and~\ref{Prelim-thm} also a cheapest ticket can be computed with a shortest path algorithm as argued in Section~\ref{sec-distance}.
Due to~\eqref{eq-comb-cheapest}, this can also be done with Algorithm~\ref{algo-combined}.

\begin{corollary}
	For a bounded distance tariff, a cheapest standard ticket and a cheapest ticket can be computed in polynomial time.
\end{corollary}

\subsection{Basic Zone Tariff Combined with a Short-distance Tariff}
\label{section-zone-short}

A ticket option that we have not studied yet and which is only implemented in combination with another fare structure is the \emph{short-distance tariff}.
Given a fare structure, e.g., a flat tariff or a zone tariff, a short-distance tariff adds a new ticket option for very short journeys with respect to the length and the number of stations.
It is designed to make such short journeys more attractive.
Since a short-distance ticket can only be bought for a subset of all paths, we define the short-distance tariff as a formal fare structure, i.e., we allow an infinite price which is not possible in practice.

For the following definition, we denote by $s(W)$ the number of stations of a path $W$ except
the start station, i.e., its number of edges.
As before, $l(W)$ is the length of a path $W\!$.

\begin{definition}
	Let a PTN be given.
	A (formal) fare structure $p$ is called a \emph{short-distance tariff} w.r.t.\ a price $P_S\in \R_{\geq 0}$ and upper bounds $\smax \in \N_{\geq 1}\cup \{ \infty \}$ and $\lmax \in \R_{>0} \cup \{ \infty \}$, where at least one is finite, if
	\[ p(W) = \begin{cases}
		P_S 			& \text{if } s(W) \leq \smax \text{ and } l(W)\leq \lmax,\\
		\infty		 	& \text{if } s(W) > \smax \text{ or } l(W)> \lmax
	\end{cases} \]
	for all paths $W \in \W$.
\end{definition}

Here, $\smax$ is an upper bound on the number of stations and $\lmax$ is an upper bound on the length of the path that are allowed for using the short-distance tariff.
We call a path $W$ a \emph{short-distance path} if $s(W)\leq \smax$ and $l(W)\leq \lmax$.
This results in the reformulation
\[ p(W) = \begin{cases}
	P_S 			& \text{if } W \text{ is a short-distance path,}\\
	\infty		 	& \text{otherwise.}
\end{cases} \]

We allow $\smax = \infty$ or $\lmax= \infty$ in order to represent the situation where a short-distance path is only restricted by the number of stations or its length, but not by both.
We first observe what cheapest paths for a short-distance tariff look like. 

\begin{lemma}\label{sd Lemma Eig3}
	Let $p$ be a short-distance tariff and $x,y \in V\!$.
	If there is a short-distance path from $x$ to $y$, then this is a cheapest path.
	If there is none, then any $x$-$y$-path is a cheapest path.
\end{lemma}

We next investigate the no-stopover and the no-elongation property.

\begin{lemma}\label{sd Lemma Eig1}
	For all bounds $\smax$ and $\lmax$, there is a PTN so that the induced short-distance tariff with respect to any price $P_S$ does not satisfy the no-stopover property.
\end{lemma}
\begin{proof}
	Let $P_S$ be an arbitrary short-distance price.
	Consider a PTN with a path $W$ with $s(W) > \smax$ or $l(W) > \lmax$ that can be decomposed into two short-distance paths~$W_1$ and $W_2$, i.e., $s(W_i)\leq \smax$ and $l(W_i) \leq \lmax$ for all $i \in \{1,2\}$.
	Then the induced short-distance tariff $p$ does not satisfy the no-stopover property since $p(W) = \infty$, but $p(W_1)+p(W_2) = 2 P_S < \infty$.
\end{proof}

\begin{theorem} \label{sd Thm Eig2}
	If $p$ is a short-distance tariff, then the no-elongation property is fulfilled.
\end{theorem}
\begin{proof}
	Let a path $W=(x_1,\ldots,x_n) \in \W$ with $n \geq 2$ be given.
	If $W$ is not a short-distance path, then $p([x_1,x_{n-1}]) \leq p(W)$.
	If $W$ is a short-distance path, then its subpath $[x_1,x_{n-1}]$ is also a short-distance path because it holds that ${s([x_1,x_{n-1}]) < s(W)}$ and $l([x_1,x_{n-1}]) < l(W)$.
	Hence, we have that $p([x_1,x_{n-1}]) \leq p(W)$. 
\end{proof}

In order to compute a cheapest path, we need to identify if a short-distance path between two stations exists, i.e., we look for a \emph{shortest weight constrained path}.
This problem, also known as \emph{restricted shortest path} has been studied extensively, e.g., \cite{Jok66,Jaf84,Has92,LoRaz01,DuBo03}.
It is known to be NP-complete \cite{GaJo}.

However, in our special case, the weight represents the number of stations on a path, which coincides with the number of edges.
Hence, we have a unit weight for all edges $e \in E$.
This is the crucial factor which allows a polynomial time algorithm for checking if there is a path which satisfies the requirements of the short-distance tariff by a modification of the Bellman-Ford algorithm, as mentioned, e.g., in \cite{ArkMitPiat}.
Iteratively, we calculate the distance from the start node to every other node in the graph using at most $s \leq \smax$ edges. Then we can check if there is a path with at most $\smax$ edges that has a length of at most $\lmax$.
For the sake of completeness, the \hyperref[Alg sd]{algorithm} can be found in the \hyperref[appendix]{appendix}.
\medskip

We finally look at the combination of the basic zone tariff with a short-distance tariff. 
The idea is to make traveling on short routes less expensive.
In combination with a basic zone tariff, the short-distance tariff is especially relevant for short paths which cross zone borders.

For a given PTN, let a basic zone tariff $p_1$ with an increasing price function $P$ and a short-distance tariff~$p_2$ with a price $P_S \in \R_{\geq 0}$, and upper bounds $\smax \in \N_{\geq 1} \cup \{\infty\}$ and ${\lmax \in \R_{> 0}\cup \{\infty\}}$ be given.
With these data, we construct the \emph{combined fare structure of a basic zone tariff and a short-distance tariff (ZSD)} $p$ as in Definition~\ref{combined Definition}, given by ${p(W) = \min \{p_1(W), p_2(W)\}}$ for all $W \in \W$.
In this case $p$ resolves to
\begin{equation} \label{formula 1}
	p(W) = \begin{cases}
		\min \{ P_S, P(z(W))\} 			& \text{if } W \text{ is a short-distance path,}\\
		P(z(W))							& \text{otherwise.}
	\end{cases}
\end{equation} 

For a ZSD, it is important how to choose $P_S$.
If $P(k) \leq P_S$ for all $k \in \N_{\geq 1}$, then the short-distance tariff would never be applied and the ZSD is a basic zone tariff.
Hence, from now on, we assume that there is some $K \in \N_{\geq 1}$ such that $P_S<P(k)$ for all $k \geq K$.
If $P_S < P(1)$, then the short-distance tariff can be beneficial when zone borders are crossed, but also for short trips within a single zone.
On the other hand, if $P_S > P(1)$, it will never be used within a single zone, but only if sufficiently many zones are traversed on a path.
Therefore, by choosing the price $P_S$, the options for use of the short-distance tariff can be limited further.
This can easily be seen by restating formula~\eqref{formula 1} for~$p$ in the following way:

Let $K \in \N$ such that $P(K) \leq P_S < P(K+1)$ where we define $P(0) \defeq 0$ in order to formally cover the case $P_S \leq P(1)$. 
Then
\begin{equation}\label{formula 2}
	p(W) = \begin{cases}
		P_S			& \text{if } z(W)>K \text{ and } s(W) \leq \smax \text{ and } l(W)\leq \lmax,\\
		P(z(W)) 	& \text{if } z(W) \leq K \text{ or } s(W) > \smax \text{ or } l(W)> \lmax
	\end{cases}
\end{equation}
is a reformulation of~\eqref{formula 1}.
In particular, it demonstrates that a path needs to traverse more than $K$ zones in order that it is beneficial to use the short-distance tariff.

\begin{theorem}\label{zone short Thm Eig1 neu}
	Let bounds $\smax$ and $\lmax$, an increasing price function $P$ and a price $P_S$ be given, and let~$K$ be as in \eqref{formula 2}.
	All induced ZSDs satisfy the no-stopover property if and only if the following conditions hold:
	\begin{enumerate}
		\item \label{condition 1} $P(k) \leq P(i) + P(k-i+1)$ for all $k \geq 1$ and $i \in \{1,\ldots, k\}$,
		\item \label{condition 2} $P(k) \leq 2 P_S$ for all $k \geq 2K+1$,
		\item \label{condition 3} $P(k) \leq P(i) + P_S$ for all $k \geq K+1$ and $i \in \{ 1, \ldots, k-K \}$,
		\item \label{condition 4} if $\smax >1$, then $P_S \leq P(i) + P(k-i+1)$ for all $k\in \{K+1, \ldots, 2K-1\}$ and for all ${i \in \{k-K+1,\ldots, K\}}$ (note that the condition is empty for $K\in \{0,1\}$). 
	\end{enumerate}	
\end{theorem}
\begin{proof}
	In this proof, we make use of the representation of the fare structure $p$ as in~\eqref{formula 2}.
	First, we show that there is a PTN for which the induced ZSD does not satisfy the no-stopover property if one of the conditions is not fulfilled.
	For the following examples, we assume that $\lmax<\infty$.
	In each case, we consider the situation depicted in Figure~\ref{Figure zone short Thm Eig1 neu} (again omitting virtual nodes which in particular do not count as stations for the short-distance tariff) with specific edge lengths so that the no-stopover property is not satisfied.
	However, the examples can be adapted to the case that $\lmax=\infty$, i.e., $\smax<\infty$, by using unit lengths $l(x_1,x_2)=l(x_2,x_3)$ and adding sufficiently many stations along both edges.
	In the following, we show that the compound ticket $((x_1,x_2),(x_2,x_3))$ is cheaper than the standard ticket of the path $(x_1,x_2,x_3)$.
	\begin{enumerate}
		\item Assume that condition~\ref{condition 1} is not satisfied for some $k$ and $i\leq k$, which means that we have ${P(k)>P(i)+P(k-i+1)}$.
		Let $l(x_1,x_2) = l(x_2,x_3) = \lmax +1$.
		Then the basic zone tariff is used for all three paths.
		
		\item Assume that condition~\ref{condition 2} is not satisfied for some $k\geq 2K+1$, which means that ${P(k) > 2P_S}$.
		Let ${l(x_1,x_2) = l(x_2,x_3) = \lmax}$, yielding ${l(x_1,x_2,x_3)=2\lmax >\lmax}$, and let $i=K+1$.
		Then it holds that $z((x_1,x_2))= K+1$ and 
		${z((x_2,x_3)) = k-i+1 \geq (2K+1)-(K+1)+1 = K+1}$.
		Also, ${s((x_1,x_2))=s((x_2,x_3))=1\leq \smax}$.
		Hence, the zone price $P(k)$ is used for $(x_1,x_2,x_3)$, but the short-distance price $P_S$ is applied to both subpaths.
		
		\item Assume that condition~\ref{condition 3} is not satisfied for some $k\geq K+1$ and $i\leq k-K$, i.e., ${P(k) > P(i)+P_S}$.
		Let ${l(x_1,x_2) = \lmax+1}$ and $l(x_2,x_3) = \lmax$.
		Then we have 
		$z((x_2,x_3)) = k-i+1 \geq k-(k-K)+1 = K+1$
		and $s((x_2,x_3))=1 \leq \smax$.
		Hence, the zone prices $P(k)$ and $P(i)$ are used for $(x_1,x_2,x_3)$ and $(x_1,x_2)$, but the short-distance price $P_S$ is applied to $(x_2,x_3)$.
		
		\item Assume that condition~\ref{condition 4} is not satisfied for some $k \in \{K+1, \ldots, 2K-1\}$ and ${i \in \{k-K+1,\ldots, K\}}$, i.e., $P_S > P(i)+P(k-i+1)$.
		Let $l(x_1,x_2) = l(x_2,x_3) = \frac{\lmax}{2}$, which yields $l((x_1,x_2,x_3)) \leq \lmax$.
		Also, $s((x_1,x_2,x_3))=2\leq \smax$ by assumption.
		We have $z((x_1,x_2))= i \leq K$ and $z((x_2,x_3)) = k-i+1 \leq k-(k-K+1)+1=K$.
		Therefore, the short-distance price $P_S$ is used for the path $(x_1,x_2,x_3)$ because ${z((x_1,x_2,x_3))=k>K}$ by choice of $k$, but the zone prices $P(i)$ and $P(k-i+1)$ are applied to the subpaths since they traverse at most $K$ zones.
	\end{enumerate}
	\begin{figure}[t]
		\centering
		\includegraphics[scale=1]{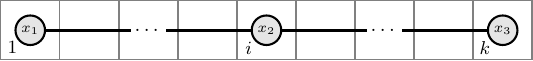}
		\caption{PTN with zones for Theorem~\ref{zone short Thm Eig1 neu}.}
		\label{Figure zone short Thm Eig1 neu}	
	\end{figure}
	\medskip
	
	Now we suppose that all the conditions hold and consider a ZSD $p$.
	Let $W \in \W$ be a path with a corresponding compound ticket $(W_1, W_2)$. We set $k \defeq z(W)$, $k_1 \defeq z(W_1)$ and $k_2 \defeq z(W_2)$, so $k_1+k_2=k+1$.
	We make a case distinction whether a path is a short-distance path and whether it is assigned the short-distance price.
	\begin{enumerate}
		\item \label{a}If neither $W$ nor $W_1, W_2$ are short-distance paths, then the basic zone tariff is applied. The no-stopover property holds for $W$ by condition~\ref{condition 1}.
		
		\item If $W_1$ is a short-distance path, but $W_2$ and $W$ are not, then we distinguish two cases:
		\begin{itemize}
			\item $k_1\leq K$: Then $W_1$ is assigned the price of the basic zone tariff, and the no-stopover property holds for $W$ by condition~\ref{condition 1}.
			
			\item $k_1\geq K+1$: The short-distance tariff is applied for $W_1$. Also $k\geq k_1 \geq K+1$ and ${k_2 = k-k_1+1 \leq k -(K+1)+1 = k-K}$. Thus, the no-stopover property holds for $W$ by condition~\ref{condition 3} with $i = k_2$.
		\end{itemize}
	\end{enumerate}
	For the cases where $W_1$ and $W_2$ or all paths are short-distance paths, we proceed analogously. 
	These remaining cases can be found in \cite{Urban20}.
	This shows that the no-stopover property is satisfied.
\end{proof}

\begin{corollary} \label{zone short lemma Eig2}
	If $p$ is a ZSD, then the no-elongation property is satisfied.
\end{corollary}
\begin{proof}
	Because basic zone tariffs with an increasing price function and short-distance tariffs satisfy the no-elongation property by Theorems~\ref{zone-elongation} and~\ref{sd Thm Eig2}, it is also satisfied for a ZSD by Theorem~\ref{combined Thm Eig2}.
\end{proof}

For a ZSD, a cheapest path is either a cheapest path as in the basic zone tariff or a short-distance path.
In order to find one, we detail Algorithm~\ref{algo-combined} and apply the algorithms for the short-distance tariff and the one for the basic zone tariff, and then we decide which path is cheaper.

\begin{algorithm}
	\caption{ZSD: finding a cheapest path.}
	\label{Alg zone short}
	\SetKwInOut{Input}{Input}
	\SetKwInOut{Output}{Output}
	\SetKwFunction{sdp}{SDP}
	\SetKw{AND}{and}
	
	\Input{PTN $(V,E)$, finite upper bounds $\smax$, $\lmax$, a price function $P\!$, a short-distance price $P_S$, two stations $x,y \in V$}
	\Output{$x$-$y$-path $W$}
	
	Let $W_1$ be the result returned by Algorithm~\ref{Alg sd} for finding a short-distance path.\\
	Let $W_2$ be the result returned by Algorithm~\ref{Alg zone} for finding a cheapest path regarding the basic zone tariff.\\
	\eIf{$W_1$ $\neq$ None \AND $P_S < P(z(W_2))$}{
		\Return  $W_1$
	}{
		\Return $W_2$
	}
\end{algorithm}

\begin{corollary}
	Let $p$ be a ZSD.
	\begin{itemize}
		\item Algorithm~\ref{Alg zone short} computes a cheapest path and hence a cheapest standard ticket in polynomial time.
		\item If $p$ satisfies the no-stopover property (Theorem~\ref{zone short Thm Eig1 neu}), then Algorithm~\ref{Alg zone short} yields a cheapest ticket in polynomial time.
	\end{itemize}
\end{corollary}
\begin{proof}
	The claim follows from the correctness and runtime of Algorithms~\ref{algo-combined}, \ref{Alg sd}, and \ref{Alg zone}.
	The no-elongation property is satisfied by Corollary~\ref{zone short lemma Eig2}.
	If also the no-stopover property holds, then a cheapest standard ticket is also a cheapest ticket by Theorem~\ref{Prelim-thm}.
\end{proof}

\section{Conclusion}
\label{sec-conclusion}
In this paper, we have provided models for many common fare structures, studied their properties and provided polynomial algorithms for finding cheapest standard tickets, all of them based on shortest paths.
We also investigated in which cases cheapest standard tickets provide cheapest tickets, i.e., in which cases it is not possible to benefit from misusing the bought tickets. 
To this end, we defined and analyzed the no-stopover and the no-elongation property and gave sufficient conditions for them to hold.

As a further step, one can investigate speed-up techniques for shortest paths (e.g., in \cite{WagnerWilhalm07,bdgmpsww-rptn-14}) in order to make the computation of cheapest paths more efficient and to evaluate these experimentally.
Here it is particularly interesting to use the embedding of the PTN in the plane, bidirectional search and the structure of the zones (for zone tariffs).
The next step is to include the ticket price as one criterion besides other criteria that passengers might apply to choose their routes.
The most important criterion for a passenger probably is the travel time, see, e.g., \cite{BornHoppKarb16}, but also, for example, the robustness against delays of the path may be important as shown in \cite{BGMHSS11,BGMHSS12}. 
Considering several criteria can be done efficiently if ticket prices can be computed by common shortest path algorithms in the same network as the travel time, but with adapted edge weights, as in \cite{GuMueSchn}.
Aside from that, other ticket options like group or season tickets are interesting for further research.
Also planning fare structures under different criteria (such as fairness, income, low transition costs) is an interesting topic for further research as well as the integration of planning fare structures and network design.
For example, there is a strong relation between line planning and the design of zone tariffs since the number of traversed zones on a path depends on the line plan.

We finally plan to include the ticket prices in route choice models and integrate them into planning lines and timetables along the lines of~\cite{SchiSch18}, but with underlying realistic passenger behavior.

\bibliographystyle{amsalpha} 
\bibliography{waben}

\newpage
\appendix
\section*{Appendix. Algorithm for Section~\ref{section-zone-short}}
\label{appendix}
\begin{algorithm}[h]
	\caption{Short-distance tariff: finding a short-distance path.}
	\label{Alg sd}
	\SetKwInOut{Input}{Input}
	\SetKwInOut{Output}{Output}
	\SetKwFunction{sdp}{SDP}
	\SetKw{AND}{and}
	
	\Input{PTN $(V,E)$, upper bounds $\smax$, $\lmax$, two stations $x,y \in V$}
	\Output{shortest $x$-$y$-path $W$ with $s(W)\leq \smax$ and $l(W)\leq \lmax$ if one exists}
	
	\tcp{Initialization}
	$\smax \defeq \min \{\smax, \vert V \vert-1\}$,
	$\lmax \defeq \min \{\lmax, \max_{e \in E}l(e)\cdot \vert V \vert\}$\\
	\For{all $v \in V$}{
		$d_0(v) \defeq \infty$\\
		$\pi_0(v) \defeq \text{None}$
	}
	$d_0(x) \defeq 0$
	
	\vspace{\baselineskip}
	\tcp{Compute distances and predecessors (Bellman-Ford)}
	\For{$s = 1, \ldots, \smax$}{
		\For{all $v \in V$}{
			$d_{s}(v) \defeq d_{s-1}(v)$\\
			$\pi_s(v)\defeq \pi_{s-1}(v)$\\
			\For{all edges $(w,v)\in E$}{
				\If{$d_{s}(v) > d_{s-1}(w)+ l(w,v)$}{
					$d_{s}(v) \defeq d_{s-1}(w)+ l(w,v)$\\
					$\pi_s(v) \defeq w$
				}
			}	
		}	
	}	
	\vspace{\baselineskip}
	\tcp{Check if a feasible $x$-$y$-path exists and compute it if necessary}
	\eIf{$d_{\smax}(y) \leq \lmax$}{
		\tcp{Determine the path $W$ from $x$ to $y$ by backtracking the predecessors}
		$W' = [y]$ \tcp{list of all predecessors starting from $y$}
		current $\defeq$ $y$\\
		$s \defeq \smax$\\
		\While{current $\neq$ $x$}{
			current $\defeq$ $\pi_s(\text{current})$\\
			$W'$.append(current) \tcp{add \textnormal{current} to the end of $W'$}
			$s \defeq s-1$
		}
		Set $W \defeq (W')^{-1}$\\
		\Return $W$
	}{
		\Return None
	}
\end{algorithm}

\end{document}